\documentclass[11pt]{amsart}
\usepackage{amsfonts, amsmath, amssymb, amsthm, color,
  bookmark,graphicx,subfig}

\usepackage{amsmath,amsfonts,amsthm,amssymb,amscd}
\usepackage{hyperref}

\usepackage[alphabetic]{amsrefs}
\usepackage{tikz}
\usepackage{graphicx}
\usepackage{tikz-cd}
\usepackage{xcolor}
\usepackage{amsthm}

\usepackage[utf8]{inputenc}
\usepackage{hyperref}
\hypersetup{
  pdfborder={0 0 0},
  colorlinks   = true,
  urlcolor     = blue,
  linkcolor    = blue,
  citecolor   = red
}

\allowdisplaybreaks

\usetikzlibrary{patterns}

\newlength{\hatchspread}
\newlength{\hatchthickness}
\newlength{\hatchshift}
\newcommand{\hatchcolor}{}
\tikzset{hatchspread/.code={\setlength{\hatchspread}{#1}},
         hatchthickness/.code={\setlength{\hatchthickness}{#1}},
         hatchshift/.code={\setlength{\hatchshift}{#1}},
         hatchcolor/.code={\renewcommand{\hatchcolor}{#1}}}
\tikzset{hatchspread=3pt,
         hatchthickness=0.4pt,
         hatchshift=0pt,
         hatchcolor=black}
\pgfdeclarepatternformonly[\hatchspread,\hatchthickness,\hatchshift,\hatchcolor]
   {custom north west lines}
   {\pgfqpoint{\dimexpr-2\hatchthickness}{\dimexpr-2\hatchthickness}}
   {\pgfqpoint{\dimexpr\hatchspread+2\hatchthickness}{\dimexpr\hatchspread+2\hatchthickness}}
   {\pgfqpoint{\dimexpr\hatchspread}{\dimexpr\hatchspread}}
   {
    \pgfsetlinewidth{\hatchthickness}
    \pgfpathmoveto{\pgfqpoint{0pt}{\dimexpr\hatchspread+\hatchshift}}
    \pgfpathlineto{\pgfqpoint{\dimexpr\hatchspread+0.15pt+\hatchshift}{-0.15pt}}
    \ifdim \hatchshift > 0pt
      \pgfpathmoveto{\pgfqpoint{0pt}{\hatchshift}}
      \pgfpathlineto{\pgfqpoint{\dimexpr0.15pt+\hatchshift}{-0.15pt}}
    \fi
    \pgfsetstrokecolor{\hatchcolor}
    \pgfusepath{stroke}
   }

\usepackage{chngcntr}
\usepackage{apptools}

\usepackage[shortlabels]{enumitem}

\usetikzlibrary{decorations.markings}
\usetikzlibrary{patterns}

\newtheorem{theorem}{Theorem}[section]
\newtheorem{lemma}[theorem]{Lemma}
\newtheorem{proposition}[theorem]{Proposition}
\newtheorem{corollary}[theorem]{Corollary}
\newtheorem*{claim}{Claim}

\theoremstyle{definition}
\newtheorem{definition}[theorem]{Definition}
\newtheorem{convention}[theorem]{Convention}

\theoremstyle{question}
\newtheorem{ques}[theorem]{Question}

\newtheorem{conj}[theorem]{Conjecture}

\newtheorem*{conjBrown}{Brown's Conjecture}

\theoremstyle{remark}
\newtheorem{example}[theorem]{Example}
\newtheorem{remark}[theorem]{Remark}

\numberwithin{equation}{section}

\newcommand{\abs}[1]{\lvert#1\rvert}
\newcommand{\bigslant}[2]{{\raisebox{.2em}{$#1$}\left/\raisebox{-.2em}{$#2$}\right.}}

\newcommand{\mF}{\mathcal {F}}
\newcommand{\mP}{\mathcal {P}}
\newcommand{\mA}{\mathcal {A}}

\newcommand{\ugd}{\underline{\mathrm{gd}}}
\newcommand{\ucd}{\underline{\mathrm{cd}}}
\newcommand{\vcd}{{\mathrm{vcd}}}

\newcommand{\fgd}{{\mathrm{gd}_{\mF}}}
\newcommand{\fcd}{{\mathrm{cd}_{\mF}}}

\newcommand{\eg}{{\underline EG}}

\newcommand{\EFG}{E_{\mF}G}

\newcommand{\isc}{\subset}

\newcommand{\dB}[1]{\mathrm{lcd}\mathcal #1}

\newcommand{\B}{B}
\newcommand{\D}{K}

\newcommand{\Z}{\mathbb Z}
\newcommand{\mQ}{\mathcal Q}

\newcommand{\mR}{\mathcal R}
\newcommand{\mB}{\mathcal B}

\newcommand{\orb}{\mathcal{O}_{\mF}G}
\newcommand{\orbmod}{\mbox{Mod-}\mathcal{O}_{\mF}G}

\newcommand{\nathom}{\mathrm{Hom}_{\mF}}

\newcommand{\rH}{\widetilde{H}}
\newcommand{\nH}{H}
\newcommand{\pan}[2]{{#1}_{#2}}
\newcommand{\link}[2]{{#1}_{>#2}}

\newcommand{\flink}[2]{\cup_{#2 < #2'}\pan{#1}{#2'}}

\begin{document}

\tikzset{->-/.style={decoration={
  markings,
  mark=at position .65 with {\arrow{>}}},postaction={decorate}}}

\title[Invariants of simple complexes of groups]{Cohomological and geometric invariants of simple complexes of groups}

\author[N.\ Petrosyan]{Nansen Petrosyan}

\address{School of Mathematics, University of Southampton, Southampton SO17 1BJ, UK}

\email{n.petrosyan@soton.ac.uk}

\author[T.\ Prytu{\l}a]{Tomasz Prytu{\l}a}

\address{Department of applied mathematics and computer science, Technical University of Denmark, Lyngby, Denmark}

\email{tomasz.prytula@alexandra.dk}

 \thanks{Both authors were supported by the EPSRC First Grant EP/N033787/1. T. P. was supported by the EU Horizon 2020 program under the Marie Sk{\l}odowska-Curie grant agreement no. 713683 (COFUNDfellowsDTU)}

\date{\today}


\begin{abstract} 

We investigate cohomological properties of fundamental groups of strictly developable simple complexes of groups $X$.  We obtain a polyhedral complex equivariantly homotopy equivalent to $X$ of the lowest possible dimension.  As applications, we obtain a simple formula for proper cohomological dimension of $\mathrm{CAT}(0)$ groups whose actions admit a strict fundamental domain.
For any building of type $(W,S)$ that admits a chamber transitive action by a discrete group, we give  a realisation of the building of the lowest possible dimension equal to the virtual cohomological dimension of $W$. 
Under general assumptions, we  confirm a folklore conjecture on the equality of Bredon geometric and cohomological dimensions in dimension one. Lastly, we give a new family of counterexamples to the strong form of Brown's conjecture on the equality of virtual cohomological dimension and Bredon cohomological dimension for proper actions. \end{abstract}

\maketitle

\section{Introduction}\label{sec:intro}

\subsection*{Overview}

	For a finitely generated Coxeter system $(W, S)$, the Davis complex $\Sigma_W$ is a $\mathrm{CAT}(0)$ polyhedral complex on which the Coxeter group $W$ acts properly, cocompactly and by reflections. The complex $\Sigma_W$ is very useful in understanding properties of $W$ or more generally of buildings of type $(W, S)$ where it appears as an apartment.  However, the Davis complex $\Sigma_W$ does not in general produce the realisation of these buildings of the lowest possible dimension.  There is an associated  contractible  polyhedral complex $B(W,S)$ of dimension equal to the virtual cohomological dimension $\vcd W$ of the Coxeter group $W$ (except possibly when $\vcd W=2$) introduced by Bestvina in \cite{Best}. The group $W$ acts by reflections properly and cocompactly on $B(W,S)$.  Bestvina complex $B(W,S)$ is equivariantly homotopy equivalent to the Davis complex $\Sigma_W$  \cite{PePry}. 
	Therefore by replacing the apartments with $B(W,S)$ one obtains a realisation of the building of type $(W, S)$ of the lowest possible dimension. In \cite{PePry}, we  derived analogous results in the more general setting of strictly developable thin simple complexes of finite groups. In doing so we  relied on compactly supported cohomology as a convenient tool for computations.  This is certainly the norm, as compactly supported cohomology can be very useful in computations of the cohomology of a $G$-CW-complex with group ring coefficients \cite{brownwall}, \cite{Best}, \cite{HarMei}, \cite{Davbook}, \cite{DMP}.  A major drawback of this approach  however is that it  restricts one to only complexes that are locally finite.

	To resolve this difficulty, in this paper we introduce a new approach that bypasses compactly supported cohomology and thus allows us to study non-proper actions admitting a strict fundamental domain, or equivalently, simple complexes of groups whose local groups need not be finite. Given a simple complex of groups $G(\mQ)$,  we first extend the definition of the Bestvina complex to $G(\mQ)$. Our methods then directly link Bredon cohomology of the Bestvina complex associated to $G(\mathcal Q)$ with certain coefficients, and the relative integral cohomology of the panel complexes over the poset $\mQ$. This enables us to compute the cohomology of the fundamental group of $G(\mathcal Q)$, determine its cohomological dimension, and identify it with the dimension of the generalisation of Bestvina complex in this context.

  Our approach also leads to cohomological computations on more naturally occurring simple complexes of groups without the \emph{thinness} assumption. This is a standing assumption in both \cite{DMP} and \cite{PePry}. It states that the cellular structure of a complex is in a sense minimal with respect to the group action, and it is fairly restrictive. In particular, removing this assumption, allows us to investigate group actions on $\mathrm{CAT}(0)$ polyhedral complexes that admit a strict fundamental domain.	
		
		Besides aforementioned applications, another important motivation to study the generalised Bestvina complex comes from the Baum-Connes and the Farrell-Jones Conjectures (see e.g.,\ \cite{BCH}, \cite{Lucksurvey}), where it is always desirable to have models for the classifying space  of $G$ for the family of stabilisers $\mF$ of minimal dimension and cell structure (see e.g.,\ \cite{MFR} for a direct application of Bestvina complex).

    \subsection*{Statement of results}

    A \emph{simple complex of groups} $G(\mathcal Q)$ over a finite poset $\mathcal Q$ consists of a collection of groups $\{P_J\}_{J \in \mathcal Q}$ and a collection of monomorphisms $\{P_J \to P_T\}_{J \leqslant T \in \mathcal Q}$ satisfying the obvious compatibility conditions. We say that $G(\mathcal Q)$ is {\it thin} if the monomorphism $P_J \to P_T$  is an isomorphism if and only if $J=T$. The \emph{fundamental group} $G$ of $G(\mathcal Q)$ is defined as the direct limit of the system $\{P_J\}_{J \in \mathcal Q}$.  We say that $G(\mathcal Q)$ is \emph{strictly developable} if for every group $P_J$ the canonical map to the limit $G$ is a monomorphism; in this case, we identify $P_J$ with its image in $G$ and call it a \emph{local subgroup} of $G$.

	A {\it family} of subgroups of a discrete group $G$ is a collection of subgroups that is closed under conjugation and taking subgroups. Given a collection of subgroups  $\{P_J\}_{J \in \mathcal Q}$ of $G$, the family \emph{generated by the collection} $\{P_J\}_{J \in \mathcal Q}$ is the smallest family $\mF$  of subgroups of $G$ that contains all elements of $\{P_J\}_{J \in \mathcal Q}$. 

	Suppose $G(\mathcal Q)$ is strictly developable with fundamental group $G$. We say that $G(\mathcal Q)$  is {\it rigid}, if for any $J\in \mQ$ no $G$-conjugate of $P_J$ is properly contained in $P_J$. Define a {\it block} $C \subseteq \mathcal Q$ as an equivalence class of elements of $\mathcal Q$ under the relation $\sim$ generated by $J'\sim J$ if $J' \leqslant J$ and $P_{J'} \to P_J$ is an isomorphism. For a fixed $J\in \mQ$ and $g\in G$, let $\Omega^g_J$ be the subset of $\mQ$ that consists of all $U \in \mQ$ for which $P_U=g^{-1}P_Jg$ (seen as subgroups of $G$). We denote by $C_J^g\subseteq \Omega^g_J$ a block in $\Omega^g_J$. Let $\mathcal I_J$ be a complete set of representatives of  
	$$\{g\in G \;|\; g^{-1}P_J g = P_U \mbox{ for some  } U \in \mQ \}/{P_J},$$ 
	 where $P_J$ acts by left multiplication.

	Let $K = |\mathcal Q |$ denote the geometric realisation of the poset $\mathcal Q$. For a subset $\Omega \subseteq \mathcal Q$ such that $P_U=P_{U'}$ for all $U, U'\in \Omega$, define subcomplexes $K_{\Omega}$ and $K_{>\Omega}$ of $K$ as follows
	\begin{align*}
		K_{\Omega}&= \abs{ \{ V \in \mathcal Q \mid V \geqslant J \mbox{ for some } J \in \Omega\}},\\
 		K_{>\Omega}&= \abs{ \{ V \in \mathcal Q \mid V \geqslant J \mbox{ for some } J \in \Omega  \text{ and } P_V\gneq P_J\}}.
	\end{align*}
  
	The complex $K=\abs{Q}$ is an example of a \emph{panel complex} over the poset $\mQ$. For a panel complex $Y$ over $\mQ$, the \emph{Basic Construction} is a $G$-space $D(Y, G(\mathcal{Q}))$ obtained by gluing copies of $Y$ indexed by elements of $G$, according to the combinatorial information in $Y$ and $G(\mQ)$.

	We denote by $H^*_{\mF}(X; M)$ the Bredon cohomology groups of a $G$-CW-complex $X$ with respect to the family of subgroups $\mathcal F$ of $G$ with coefficients a contravariant functor $M$ from the orbit category $\orb$ to $\Z\mbox{-Mod}$. In what follows, we will restrict to coefficients $\mA_{H}=\Z[\mbox{hom}_G(-, G/H)]$ for a subgroup $H \in \mF$, and a certain refinement of $\mA_{H}$ which we denote by $\mB_{H}$. Let $\fcd G$ (resp.\ $\fgd G$) denote the Bredon cohomological (resp.\ geometric) dimension of $G$ with respect to the family $\mF$ and let $\EFG$ denote the  universal $G$-CW-complex with stabilisers in $\mF$.
	If $\mF$ is the family of all finite subgroups of $G$, then the respective notions are denoted by $\ucd G$, $\ugd G$, and $\underbar{E}G$. 

	\begin{theorem}[Theorem~\ref{thm:textrigidgen}]\label{thm:introrigidgen}
	     Let $G(\mathcal Q)$ be a strictly developable simple complex of groups with fundamental group $G$, and let $\mF$ be the family generated by local groups. Let $X=D(\D, G(\mathcal{Q}))$ be the associated Basic Construction. For $J\in \mQ$ we have
	  \begin{equation}\label{eq:introrigidzero}
	    H^*_{\mF}(X; \mB_{P_J}) \cong \bigoplus_{g\in \mathcal I_J}\bigoplus_{C_J^g\subseteq \Omega^g_J} H^*(K_{C_J^g}, K_{>C_J^g}).
	  \end{equation}
	 
	  If $G(\mathcal Q)$ is rigid and $X$ is a model for $\EFG$, then 
	  \begin{equation}\label{eq:introrigid}
	    \fcd G=\mathrm{max}\{n \in \mathbb{N} \mid H^{n}(K_C, K_{> C}) \neq 0 \text{ for some block } C \subseteq {\mathcal Q}\}.
	  \end{equation} 
	\end{theorem}

	The rigidity assumption holds for example when the local groups are co-Hopfian, and hence in particular when they are finite. We should also remark that  the rigidity assumption on local groups in Theorem \ref{thm:introrigidgen} is not superfluous.

	Recall that an action of a group on cellular complex is  admissible, if the setwise stabiliser of each cell is also its pointwise stabiliser.  If a group $G$ acts admissibly on a simply connected cellular complex with a strict fundamental domain $Y$ then it is isomorphic to the fundamental group of a simple complex of groups formed by cells of $Y$ and their stabilisers (see Theorem~\ref{thm:scogisfund}). The following corollary of Theorem~\ref{thm:introrigidgen} is straightforward.

	\begin{corollary}
	Suppose a group $G$ acts properly and admissibly on a $\mathrm{CAT}(0)$ polyhedral complex $X$ with a strict fundamental domain $Y$. Let $\mathcal{Q}$ denote the poset of cells of $Y$ ordered by  reverse inclusion (note that we have $\abs{\mathcal Q}=K=Y'$). Then 
	  \begin{equation*}\label{eq:introrigid2}
	    \ucd 	G=\mathrm{max}\{n \in \mathbb{N} \mid H^{n}(K_C, K_{> C}) \neq 0 \text{ for some block } C \subseteq {\mathcal Q}\}.
	  \end{equation*}
	  
	\end{corollary}

	 This corollary is a generalisation of \cite[Theorem~1.2]{DMP} to  non-thin complexes of groups. We remark that non-thinness of a complex of groups resulting from the $G$-action on $X$ is generic, e.g.,\ in many cases in order to obtain an admissible action one takes the barycentric subdivision which results in a non-thin complex.\medskip
	 
	 To obtain formula~\eqref{eq:introrigid} of Theorem~\ref{thm:introrigidgen}, we prove the following general result.

	\begin{theorem}[Theorem~\ref{thm:stab}]\label{intro:stab} 
	  Let $G$ be a group and $\mF$ be a family of subgroups. Suppose that $X$ is a cocompact model for $\EFG$. Then 
	  $$\fcd G=\mathrm{max}\{k \in \mathbb{N} \mid H^{k}_{\mF}(X, {\mA}_H) \neq 0  \text{ for some cell stabiliser } H \}.$$
	  Moreover, if $H^{n}_{\mF}(G; \mA_L) \neq 0$ for $n=\fcd G$ and $L\in \mF$, then there exists a cell stabiliser $H\leqslant L$ such that
	  $H^{n}_{\mF}(G; \mA_H) \neq 0$.
	\end{theorem}

	  Note that, under the assumptions of the theorem, there are only finitely many conjugacy classes of stabilisers. Thus the theorem reduces the computation of the Bredon cohomological dimension of a given group into a computation of finitely many cohomology groups. Theorem~\ref{intro:stab}  together with \cite[Theorem~2.4]{DMP} give  us the following strengthening of \cite[Theorem~1.1]{DMP}.

	\begin{corollary}[Corollary~\ref{cor:support}]\label{cor:compactsupport}
		Let $X$ be a $G$-CW-complex that is a cocompact model for $\eg$. Then $$\ucd G=\mathrm{max}\{k \in \mathbb{N} \mid H^{k}_{c}(X^H, X_{\mathrm{sing}}^H) \neq 0  \text{ for some cell stabiliser } H \},$$
		where $X_{\mathrm{sing}}^H \subseteq X^H$ consists of all points whose stabiliser strictly contains $H$.
	\end{corollary}

	 Another application of Theorem~\ref{thm:introrigidgen} is the construction of new counterexamples to the strong form of Brown's conjecture via the notion of \emph{reflection-like actions}. Here the removal of thinness assumption is the key to obtaining a systematic approach to constructing such examples. Reflection-like actions are generalisations of groups acting by reflections on Euclidean spaces.  The precise definition  and examples can be found in  Section~\ref{sec:reflikeactions}.

	\begin{theorem}[Theorem~\ref{thm:reflikemain}]\label{thm:intoreflikemain}
	  Let $F$ be a finite group admitting a reflection-like action on a compact, connected,  flag simplicial complex $L$ of dimension $n \geqslant 1$. Let $W_L$ be the right-angled Coxeter group associated to $L$ and $G=W_L\rtimes F$ be the associated semi-direct product. Suppose that $\nH^n(L)=0$. Then  
	  \begin{equation*}
	    \vcd G\leqslant n \quad \text{and} \quad \ucd G=n+1.
	  \end{equation*}
	\end{theorem}

	We refer the reader to Examples \ref {ex:dihedralmoore}, \ref{ex:twocoprimemoorespaces} for a specific construction of complexes $L$ satisfying the hypothesis of Theorem \ref{thm:intoreflikemain} via products of dihedral group actions on 2-dimensional Moore spaces.

	Observe that as long as the complex of groups $G(\mQ)$ is thin, Theorem~\ref{thm:introrigidgen} implies that the Bredon cohomological dimension of $G$ depends only on the poset structure of $\mathcal Q$. We show that for a strictly developable thin simple complex of groups, there is a model for $\EFG$ of the smallest possible dimension and a simple cell structure. The model is given as the Basic Construction where one replaces panel complex $K$ with the so-called \emph{Bestvina complex} $B$.

	\begin{theorem}\label{thm:intromain} Let $G(\mathcal Q)$ be a strictly developable thin complex of groups over a poset $\mathcal{Q}$ with fundamental group $G$ and  let $\mF$ be the family generated by the local groups. Then 
	  \begin{enumerate}[label=(\roman*)]
	    
	    \item \label{it:HoeqPlusFormulaCohomology} the standard development $D(\D, G(\mathcal{Q}))$  and the Bestvina complex $D(\B, G(\mathcal{Q}))$ are $G$-homotopy equivalent, and
	    $$H^*_{\mF}(D(\D, G(\mathcal{Q})); \mB_{P_J}) \cong \bigoplus_{g\in \mathcal I_J}\bigoplus_{U\in \Omega^g_J} \rH^{*-1}(K_{>U}).$$

	    \item \label{it:mainintroCohDim} if  $D(\D, G(\mathcal{Q}))$  is a model for $\EFG$, then $D(\B, G(\mathcal{Q}))$ is a cocompact model for $\EFG$ satisfying
	    \[\mathrm{dim}(D(B, G(\mathcal{Q}))) = \left\{ 
	      \begin{array}{lcr}
	        \fcd G		  & \text{ if } & \fcd G \neq 2\\
	        2 \text{ or } 3 & \text{ if } & \fcd G =2\\

	      \end{array} \right.\]
	    and 
	      \begin{equation}\label{eq:introdim}
	        \fcd G=\mathrm{max}\{n \in \mathbb{N} \mid \rH^{n-1}(K_{> J}) \neq 0 \text{ for some } J \in \mathcal Q \}.
	      \end{equation}
	    
	    \end{enumerate}

	\end{theorem}

	Since buildings are CAT(0) and chamber transitive actions on them are thin (see Lemma~\ref{lem:buildingmin}), they are ideally suited for applying Theorem~\ref{thm:intromain}.

	\begin{corollary} \label{cor:introbuilding}
	  Let $G$ be a group acting chamber transitively on a building  of type $(W,S)$. Let $G(\mQ)$ be the associated simple complex of groups and let $\mF$ be the family generated by the stabilisers. Then $D(B, G(\mathcal{Q}))$ is a realisation of the building (and thus a cocompact model for $\EFG$) of dimension

	\[\mathrm{dim}(D(B, G(\mathcal{Q}))) = \left\{ \begin{array}{lcr}
	      \vcd W         & \text{ if } & \vcd W\neq 2,\\
	      2 \text{ or } 3 & \text{ if } & \vcd W=2. \\

	      \end{array} \right.\]
	      Moreover,  $$H^*_{\mF}(G; \mB_{P_J}) \cong \bigoplus_{g\in \mathcal I_J}\bigoplus_{U\in \Omega^g_J} \rH^{*-1}(K_{>U})$$
	      and $$\fcd G=\vcd W= \mathrm{max}\{n \in \mathbb{N} \mid \rH^{n-1}(K_{> J}) \neq 0 \text{ for some } J \in \mathcal Q \}.$$
	\end{corollary}
	
	The formula for Bredon cohomological dimension in Corollary \ref{cor:introbuilding}  extends \cite[Corollary~1.4]{DMP} from finite to arbitrary stabilisers. As a consequence of Corollary~\ref{cor:introbuilding}, we obtain one of the main results of \cite{Harland}.

  \begin{corollary}[Corollary~\ref{cor:textbuilding_3}]\label{cor:introbuilding_3} 
	  Let $G$ be a virtually torsion-free group acting chamber transitively on a building of type $(W,S)$. Then $$\vcd G\leqslant \vcd W + \max \{\vcd P \;|\; P \mbox{ is a special parabolic subgroup of  } G\}.$$
	\end{corollary}

  We point out that in \cite[Theorem 2.8]{Harland} it is proven that, under the assumptions of Theorem \ref{thm:intromain}, the dimension of Bestvina complex is minimal among $G$-complexes which admit a strict fundamental domain with all acyclic panels. Theorem \ref{thm:intromain}\ref{it:mainintroCohDim} is stronger, as it states that the dimension of the Bestvina complex is minimal among all possible models for $\EFG$ (except the case where $\fcd G=2$).\smallskip

	The next corollary lists equivalent conditions for fundamental groups of strictly developable thin simple complexes of groups to act on trees with the prescribed family of stabilisers.

	\begin{corollary}[Theorem~\ref{thm:deformationretraction}]\label{cor:introdeformationretraction}
	  Let $G(\mathcal Q)$ be a strictly developable thin simple complex of groups over a poset $\mathcal{Q}$ with the fundamental group $G$ and let $\mF$ be the family generated by local groups. Suppose that $D(\D, G(\mathcal{Q}))$ is a model for $\EFG$. Then the following are equivalent:
	  \begin{itemize} 
	    \item[(i)] $D(B, G(\mathcal{Q}))$ is a tree and an equivariant deformation retract of $D(K, G(\mathcal{Q}))$.
	    \item[(ii)] $ \fcd G\leqslant 1$.
	    \item[(iii)] $H^{n}(K_{> J})= 0 \text{ for all } J \in \mathcal Q$ and $n \geqslant 1$.
	  \end{itemize}
	\end{corollary}

	The following corollary is immediate.

	\begin{corollary}\label{cor:introbuilding_2} 
	  Let $G$ be a group acting chamber transitively on a building of type $(W,S)$. The geometric realisation of the building equivariantly deformation retracts onto a tree if and only if $\vcd W  \leqslant 1$.
	\end{corollary}

	Corollary \ref{cor:introdeformationretraction}  is a generalisation of \cite[Proposition 8.8.5]{Davbook} which deals with the case when $G=W$ is a Coxeter group acting on the Davis complex. It is a special case of the following folklore conjecture.  

	\begin{conj}\label{conj:folklore}
	  Let $G$ be a group and $\mF$ be a family of subgroups. Then $ \fcd G\leqslant 1$ if and only if $G$ acts on a tree with stabilisers generating $\mF$.
	\end{conj}

This conjecture is wide open in general. When $\mF$ is the trivial family, it reduces to the classical theorem of Stallings and Swan.  For the family of finite subgroups $\mF$, it follows from Dunwoody's accessibility result \cite{Dun}. Recently, in \cite{deg}, Degrijse verified the conjecture when  $\mF$ is the family of virtually cyclic subgroups. Note that  Corollary \ref{cor:introdeformationretraction} confirms the conjecture when  $G$ admits a model for $\EFG$ with a strict fundamental domain such that the associated complex of groups is thin.

\subsection*{Organisation}

	Sections~\ref{sec:bredon} and ~\ref{sec:sdabc} have a preparatory character. In Section~\ref{sec:bredon}, we give a background on classifying spaces for families of subgroups and Bredon cohomology, and we prove Theorem~\ref{intro:stab}. In Section~\ref{sec:sdabc}, we define simple complexes of groups, the Basic Construction and Bestvina complex. We describe the procedure of thinning, and we use it to compute upper bounds for the geometric dimension of the fundamental group of a simple complex of groups.

	The next three sections form the technical core of the paper. In Section~\ref{sec:minimal}, we prove Proposition~\ref{prop:chains} which allows us to compute the Bredon cohomological dimension of a fundamental group of a thin complex of groups. In Section~\ref{sec:cohofscofgs},  we prove an analogous Proposition~\ref{prop:ggchains} for an arbitrary complex of groups. Section~\ref{sec:maintheorems} contains generalised statements and proofs of Theorems~\ref{thm:introrigidgen} and ~\ref{thm:intromain}.

	In the remaining sections we discuss applications and consequences of the main theorems. In Section~\ref{sec:defretraction}, we briefly discuss the case when $\fcd G=1$ and we give a proof of Theorem~\ref{cor:introdeformationretraction}. In Section~\ref{sec:applications}, we discuss applications of our theory to chamber transitive automorphism groups of buildings and we prove Corollary~\ref{cor:introbuilding} as well as other applications and examples.
	In Section~\ref{sec:reflikeactions}, we define reflection-like actions, establish their basic properties  and prove Theorem~\ref{thm:intoreflikemain}. We then give some examples of reflection-like actions. Finally, in Section~\ref{sec:questions} we pose and discuss some open questions.

  \subsection*{Acknowledgements}

  We thank Ian Leary and Ashot Minasyan for helpful discussions.  We also thank the referee for the thorough reading of the paper and many useful suggestions that helped improve its exposition.

\section{Classifying spaces and Bredon cohomology}\label{sec:bredon}

\subsection{Classifying spaces for families of subgroups}

Let $G$ be a countable discrete group. A \emph{family} $\mF$ of subgroups of $G$ is a collection of subgroups that is closed under conjugation and taking subgroups. 
Given a collection of subgroups $\mP$ of $G$, the \emph{family of subgroups generated by} $\mP$ is the smallest family of subgroups $\mF$  of $G$ containing all subgroups of $\mP$. 

\begin{definition}
  A collection of subgroups $\mP$ of $G$ is \emph{rigid} if for every $H \in \mP$ no $G$-conjugate of $H$ is properly contained in $H$. 
\end{definition}

Recall that a \emph{polyhedron} (or a \emph{polyhedral complex}) is a CW-complex whose attaching maps are piecewise linear. We say that the action of a group $G$ on a polyhedral (CW, simplicial) complex $X$ is \emph{admissible} if for any cell $e \subset X$ its pointwise stabiliser is equal to its setwise stabiliser. In such case we call $X$ a $G$-polyhedral ($G$-CW, $G$-simplicial) complex. A $G$-CW-complex $X$ is \emph{cocompact} (or the $G$-action on $X$ is cocompact) if $X/G$ is compact, i.e.,\ it has finitely many cells. 

\begin{definition}[Classifying space $E_{\mF}G$] 
  Given a group $G$ and a family of its subgroups $\mF$, a model for the \emph{classifying space of $G$ for the family $\mF$} denoted by $E_{\mF}G$ is a $G$-CW-complex $X$ such that:
  \begin{itemize}
    \item for any cell $e \subset X$ the stabiliser $G_e$ belongs to the family $\mF$,
    \item for any $H \in \mF$ the fixed point set $X^H$ is contractible.
  \end{itemize}
\end{definition} 

The classifying space $E_{\mF}G$ is a terminal object in the homotopy category of $G$-CW-complexes  with stabilisers in $\mF$, i.e.,~if $X$ is a $G$-CW-complex with stabilisers in $\mF$ then there exists a $G$-map $X \to E_{\mF}G$ which is unique up to $G$-homotopy. In particular, any two models for $E_{\mF}G$ are $G$-homotopy equivalent. The minimal dimension of a model for $\EFG$ is called the \emph{Bredon geometric dimension of $G$ for the family $\mathcal F$} and it is denoted by $\fgd G$.

\begin{remark}
  If $\mF$ contains only the trivial subgroup, the classifying space $E_{\mF}G$ is the  universal space for free actions, commonly denoted by $EG$. If $\mF$ consists of all finite subgroups of $G$, the classifying space $E_{\mF}G$ is called the classifying space for proper actions and it is denoted by $\underline{E}G$.
\end{remark}

\subsection{Bredon cohomology}

The \emph{orbit category} $\orb$ is a category defined as follows: the objects are the left coset spaces $G/H$ for all $H \in \mF$ and the morphisms are all $G$-equivariant maps between the objects. Note that every morphism $\varphi: G/H \rightarrow G/P$ is completely determined by $\varphi(H)$, since $\varphi(xH)=x\varphi(H)$ for all $x \in G$. Moreover, there exists a morphism: $$G/H \rightarrow G/P : H \mapsto xP \mbox{ if and only if } x^{\scriptscriptstyle -1}Hx \leqslant P.$$ We denote the morphism $\varphi: G/H \rightarrow G/P: H\mapsto xP$  by $G/H \xrightarrow{x} G/P$ and note that it is determined by the inclusion $x^{\scriptscriptstyle -1}Hx \leqslant P$. Given $H, P\in \mF$, we denote by $\mbox{hom}_G(G/H, G/P)$ the set of morphisms from $G/H$ to $G/P$.

An \emph{$\orb$-module} is a contravariant functor $M \colon \orb \rightarrow \Z\mbox{-Mod}$. The \emph{category of $\orb$-modules}, denoted by $\orbmod$, is the category whose objects are $\orb$-modules and whose morphisms are natural transformations between these objects. The set of morphisms between $M, N \in \orbmod$ is denoted by $\mathrm{Hom}_{\mF}(M, N)$.

A sequence \[0\rightarrow M' \rightarrow M \rightarrow M'' \rightarrow 0\]
in $\orbmod$ is called {\it exact} if it is exact after evaluating in $G/H$ for each $H \in \mF$.  For any  $P\in \mF$, the $\orb$-module $\mA_P=\Z[\mbox{hom}_G(-, G/P)]$ is a  free object in $\orbmod$. A module $F\in\orbmod$ is free if and only if $F\cong \bigoplus_{\alpha \in I} \mA_{P_{\alpha}}$ for some collection $I$ of not necessarily distinct subgroups $P_{\alpha}\in \mF$. We will say that $F$ is {\it based} at the elements $P_{\alpha}\in \mF$, $\alpha \in I$.

The \emph{$n$-th Bredon cohomology group of $G$} with coefficients $M \in \orbmod$ is by definition
\[ \mathrm{H}^n_{\mathcal{F}}(G,M)= \mathrm{Ext}^{n}_{\orb}(\underline{\Z},M), \]
where $\underline{\Z}$ is the functor that maps all objects to $\Z$ and all morphisms to the identity map. The \emph{Bredon cohomological dimension of $G$} is defined to be 
\[ \fcd G = \sup\{ n \in \mathbb{N} \ | \  \mathrm{H}^n_{\mathcal{F}}(G,M)\neq 0 \text{ for some } M \in \orbmod \}. \]
Given a $G$-CW-complex $X$, an $\orb$-module \[{C}^{\mF}_{n}(X)(-) \colon \mathcal O_{\mathcal{F}}G \to \Z\mbox{-Mod}\] is defined as \[{C}^{\mF}_{n}(X)(G/H) = C_{n}(X^H),\] where $C_{\ast}(-)$ denotes the cellular chains. Note that, in this way, the augmented cellular chain complex of any model for $\EFG$ yields a free resolution of $\underline{\Z}$ which can then  be used to compute $\mathrm{H}_{\mF}^{\ast}(G,-)$. It follows that $\fcd G \leqslant \fgd G$.\medskip

We now consider the situation when  $G$ admits a cocompact model for $\EFG$. In  this case, Bredon cohomology commutes with arbitrary direct sums of coefficient modules (see e.g.,\ \cite[Proposition 5.2]{MarNuc}) and  one obtains the following proposition which is standard (see e.g.,\ \cite[Equation 2]{DMP}).

\begin{proposition}\label{prop:free} 
  Suppose that $X$ is a cocompact model for $\EFG$. Then 
  \[\fcd G=\sup\{k \in \mathbb{N} \mid H^{k}_{\mF}(X, \mA_H) \neq 0  \text{ for some } H\in \mF \}.\]
\end{proposition}

\noindent Below we derive a strengthening of Proposition~\ref{prop:free}, which is a key ingredient in the proof of Theorem \ref{thm:introrigidgen}.

\begin{theorem}\label{thm:stab} 
    Suppose that $X$ is a cocompact model for $\EFG$. Then 
    $$\fcd G=\mathrm{max}\{k \in \mathbb{N} \mid H^{k}_{\mF}(X, {\mA}_H) \neq 0  \text{ for some cell stabiliser } H \}.$$
    Moreover, if $H^{n}_{\mF}(G; \mA_L) \neq 0$ for $n=\fcd G$ and $L\in \mF$, then there exists a cell stabiliser $H\leqslant L$ such that
    $H^{n}_{\mF}(G; \mA_H) \neq 0$.
    \end{theorem}

\begin{proof}
  The chain complex $C_i^{\mF}(X)$ forms a resolution of $\underline{\Z}$ of finite length by finitely generated free  $\orb$-modules. Let $P=\ker\{C_{n-1}^{\mF}(X)\to C_{n-2}^{\mF}(X)\}$. Then $P$ is projective and $$0\to P\to C_{n-1}^{\mF}(X)\to \dots \to C_{0}^{\mF}(X)\to \underline{\Z}\to 0$$ is exact. By applying the Bredon analog of Schanuel's Lemma \cite[VIII.4.4]{brownco} to the above two resolutions, it follows that there is a finitely generated free  $\orb$-module $F$ based at stabilisers of the action of $G$ on $X$ such that $P\oplus F$ is  a finitely generated free $\orb$-module again based at stabilisers of the action of $G$ on $X$. We can define the resolution $(D^{\mF}_{\ast}, \partial_{\mF})$ of $\underline{\Z}$ by finitely generated free $\orb$-modules
  \[D^{\mF}_i = \left\{ \begin{array}{ll}
                                                     C_i^{\mF}(X) & i\leqslant n-2, \\
                          \displaystyle{C_{n-1}^{\mF}(X)\oplus F} & i=n-1,\\
                                         \displaystyle{P\oplus F} & i=n,\\
                                                                0 & i>n. 
      \end{array} \right.\]

\noindent Since $X$ is cocompact, by Proposition~\ref{prop:free} there exists $L\in \mF$, such that $H^{n}_{\mF}(X, \mathcal{A}_L) \neq 0$. Then $H^{n}_{\mF}(D^{\mF}_{\ast}, \mA_L) \neq 0$, which means that 
the co-boundary map 
$$\delta^L_{\mF}:\mathrm{Hom}_{\mF}(D_{n-1}^{\mF}, \mA_L)\to \mathrm{Hom}_{\mF}(D_{n}^{\mF}, \mA_L)$$
is not onto. Rewriting this more explicitly using the Yoneda Lemma, we have
$$\delta^L_{\mF}:\sum_{i=1}^k  \Z[\hom_G(G/{G_{\tau_i}}, G/{L})]\to \sum_{j=1}^l  \Z[\hom_G(G/{G_{\sigma_j}}, G/{L})]$$
is not onto. This implies that there exists a stabiliser $G_{\sigma}$ of some cell  $\sigma$ so that the generator $(G/{G_{\sigma}}\xrightarrow{x} G/L)$ of the $n$-th co-chain group is not in the image of $\delta^L_{\mF}$. Denote $H=x^{-1}G_{\sigma}x\leqslant L$. This inclusion induces  an $\orb$-module map $\mA_H\to \mA_L$ which in turn induces a map of co-chain complexes $$\Delta_{\ast}:\mathrm{Hom}_{\mF}(D_{\ast}^{\mF}, \mA_H)\to \mathrm{Hom}_{\mF}(D_{\ast}^{\mF}, \mA_L)$$
 such that  $$\Delta_n(G/{G_{\sigma}}\xrightarrow{x} G/{H})=(G/{G_{\sigma}}\xrightarrow{x} G/{L}).$$ By the commutativity  $\delta^L_{\mF}\circ \Delta_{n-1}= \Delta_{n}\circ \delta^H_{\mF}$, we obtain that $(G/{G_{\sigma}}\xrightarrow{x} G/{H})$ is not in the image of $\delta^H_{\mF}$. Therefore, $$\delta^H_{\mF}:\mathrm{Hom}_{\mF}(D_{n-1}^{\mF}, \mA_H)\to \mathrm{Hom}_{\mF}(D_{n}^{\mF}, \mA_H)$$ is not onto which shows that $H^{n}_{\mF}(X, \mA_H)=H^{n}_{\mF}(D^{\mF}_{\ast}, \mA_H) \neq 0$.  
\end{proof}

Define a subset $\mbox{isom}_G(G/L, G/S) \subseteq \hom_G(G/L, G/{S})$ by $$\mbox{isom}_G(G/L, G/S)=\{\varphi:G/L \rightarrow G/S: L\mapsto xS \; | \; x^{-1}Lx=S\}.$$ 
Define an $\orb$-module $\mB_S$  by
$$ \mB_S(G/L)= \left\{ \begin{array}{lcr}

      \Z[\mbox{isom}_G(G/L, G/S)]         & \text{ if } & L=_G S,\\
      0 & \text{ if } &L\ne_G S. \\
   \end{array} \right.$$
      where $L=_G S$ means that $L$ and $S$ are conjugate in $G$.  For each $$(\varphi:G/L \xrightarrow{x} G/S)\in \mbox{isom}_G(G/L, G/S),$$ we set
 $$ \mB_S(\theta: G/H\xrightarrow{y} G/L)(\varphi)= \left\{ \begin{array}{lcr}

         (\varphi\circ\theta: G/H\xrightarrow{yx} G/S)     & \text{ if } & y^{-1}Hy=L,\\
      0 & \text{ if } &y^{-1}Hy\ne L. \\
      \end{array} \right.$$
which is an element in  $\mB_S(G/H)$. It is not difficult to check that $\mB_S$ is well-defined.

\begin{corollary}\label{cor:B_module} 
  Suppose that $X$ is a cocompact model for $\EFG$ and that the collection of cell stabilisers is rigid. Then 
  $$\fcd G=\mathrm{max}\{k \in \mathbb{N} \mid H^{k}_{\mF}(X, \mB_P) \neq 0  \text{ for some cell stabiliser } P \}.$$
\end{corollary}

\begin{proof}  First, note that the cocompactness of $X$ implies that the set of conjugacy classes of cell stabilisers is finite.
  By Theorem \ref{thm:stab} there exists $P \in \mathcal{F}$ that is a stabiliser of a cell in $X$ such that  $H^{n}_{\mF}(X, \mA_{P}) \neq 0$ where $\fcd G=n$. By the rigidity of stabilisers and iteration of Theorem \ref{thm:stab}, we can assume that there exists such $P$ that does not contain a proper subgroup $S$ such that $H^{n}_{\mF}(X, \mA_{S}) \neq 0$. \noindent Observe that also by the rigidity for $H=_G P$ we have 
  $$\hom_G(G/H, G/{P})= \mbox{isom}_G(G/H, G/{P}).$$ 
  \noindent Again using rigidity, we can define an $\orb$-submodule $\mathcal C_{P}$ of $\mA_P$ by
  $$\mathcal C_{P}(G/H)= \left\{ \begin{array}{lcr}

     0    & \text{ if } & H =_G P,\\
       \Z[\hom_G(G/H, G/{P})] & \text{ if } &H\ne_G  P. \\

      \end{array} \right.$$

  \noindent  Considering the long exact sequence of the resulting short exact sequence 
  $$0\to \mathcal C_{P}\to \mA_{P}\to \mB_{P}\to 0,$$
  we obtain that either  $H^{n}_{\mF}(X, \mathcal C_{P}) \neq 0$ or $H^{n}_{\mF}(X, \mB_{P}) \neq 0$.  Considering a module that is a free cover of $\mathcal C_{P}$ consisting of free modules based at proper subgroups  of $P$, shows that if $H^{n}_{\mF}(X, \mathcal C_{P}) \neq 0$, then $H^{n}_{\mF}(X, \mA_{S}) \neq 0$ for some $S\lneq P$ which violates the minimality assumption on $P$. Hence, $H^{n}_{\mF}(X, \mB_{P}) \neq 0$. This establishes the claim.
\end{proof}

The Bredon cohomological and geometric dimensions for proper actions are denoted respectively by $\ucd G$ and $\ugd G$.

\begin{corollary}\label{cor:support}  
  Let $X$ be a $G$-CW-complex that is a cocompact model for $\eg$. Then
  \begin{equation*}
  \ucd G=\mathrm{max}\{k \in \mathbb{N} \mid H^{k}_{c}(X^H, X_{\mathrm{sing}}^H) \neq 0  \text{ for some cell stabiliser } H \},
  \end{equation*}
  where $X_{\mathrm{sing}}^H \subseteq X^H$ consists of all points whose stabiliser strictly contains $H$.
\end{corollary}

\begin{proof}
  The claim follows immediately from combining Corollary \ref{cor:B_module} and \cite[Theorem~2.4]{DMP}.
\end{proof}

\section{Simple complexes of groups}\label{sec:sdabc}

\subsection{Simple complexes of groups and the Basic Construction}

Throughout, let $\mQ$ be a finite poset. We denote by $\abs{\mathcal Q}$ the \emph{geometric realisation} of $\mathcal Q$, i.e., a simplicial complex whose simplices are chains of elements of $\mathcal Q$.

\begin{definition}[Simple complex of groups] 

  A \emph{simple complex of groups} $G(\mQ)$ over $\mQ$ consists of the following data:

  \begin{itemize}
    \item for any $J \in Q$ there is a group $P_J$ called a \emph{local group at} $J$,
    \item for any two elements $J \leqslant T$ in $\mQ$ there is a monomorphism \[\phi_{TJ} \colon P_J \to P_T,\] such that if $J \leqslant T \leqslant U$ then \[\phi_{UT} \circ \phi_{TJ}=\phi_{UJ}.\]
  \end{itemize}
\end{definition}

\begin{definition}[Simple morphism]
  Let $G(\mQ)$ be a simple complex of groups and let $G$ be a group. A \emph{simple morphism} $\psi \colon G(\mQ) \to G$ is a collection of maps $\psi_J \colon P_J \to G$ satisfying \[\psi_T \circ \phi_{TJ} =\psi_J \] for all pairs $J \leqslant T$ in $\mathcal Q$. We say that $\psi \colon G(\mQ) \to G$ is \emph{injective on local groups} if for every $J \in \mathcal{Q}$ the map $\psi_J \colon P_J \to G$ is injective. 
\end{definition}

Given a simple complex of groups $G(\mathcal Q)$, the \emph{fundamental group} $\widehat{G(\mathcal{Q})}$ of $G(\mQ)$ is the direct limit of the resulting directed system of groups 

\[\widehat{G(\mathcal{Q})}=\varinjlim_{J \in \mathcal Q}P_J, \]

Note that by the universal property of $\widehat{G(\mathcal{Q})}$ there exists a canonical simple morphism $\iota \colon G(\mathcal Q) \to \widehat{G(\mathcal{Q})}$ such that for every $J \in \mathcal Q$ the map $\iota_J  \colon P_J \to G(\mathcal Q)$ is the canonical map to the limit.\medskip

\begin{definition}[Strict developability]
	We say that a simple complex of groups $G(\mQ)$ is \emph{strictly developable} if the canonical simple morphism  $\iota \colon G(\mathcal Q) \to \widehat{G(\mathcal{Q})}$ is injective on local groups.

	Note that the strict developability is equivalent to the existence of a simple morphism $\psi \colon G(\mathcal Q) \to G$ that is injective on local groups, where $G$ is some group.\smallskip
\end{definition}

\begin{convention}\label{conv:localgroupsaresubgroups}
  If  $\psi \colon G(\mathcal Q) \to G$ is a simple morphism that is injective on local groups then for any $J \in \mathcal Q$ we identify the group $P_J$ with its image $\psi(P_J) \leqslant G$. 
\end{convention}

\begin{definition}[Panel complex] 
  A \emph{panel complex} $(X, \{X_J\}_{J \in Q})$ over $\mQ$ is a compact polyhedron $X$ together with family of subpolyhedra $\{X_J\}_{J \in Q} $ called \emph{panels} such that
  \begin{itemize}
    \item $X$ is the union of all the panels,
    \item $X_T \subseteq X_J$ if and only if $J \leqslant T$,
    \item for any two panels their intersection is either a union of panels or empty.
  \end{itemize}
\end{definition}

\begin{definition}[Standard panel complex]
  Define the panel complex $K$ over $\mQ$ as follows. Let $K= |\mQ|$ and for $J \in \mQ$ let $K_J= |\mQ_{\geqslant J}|$ where $\mQ_{\geqslant J}$ denotes the subposet of $\mQ$ consisting of all the elements greater than or equal to $J$.
\end{definition}

\begin{definition}[Basic Construction]
  Suppose that:
  \begin{itemize}
    \item $G(\mQ)$ is a strictly developable complex of groups,
    \item $X$ is a panel complex over $\mQ$,
    \item $\psi \colon G(\mathcal Q) \to G$ is a simple morphism to a group $G$ that is injective on local groups (thus for any $J \in \mathcal Q$ we identify $P_J$ with $\psi(P_J)$). 
  \end{itemize}
 For a point $x \in X$ let $J(x) \in Q$ be such that the panel $X_{J(x)}$ is the intersection of all the panels containing $x$. Define the Basic Construction $D(X, G(\mQ), \psi)$ as follows:
  \[D(X, G(\mQ), \psi)= G  \times X/ \sim\]
  where $(g_1,x_1) \sim (g_2,x_2)$ if and only if $x_1=x_2$ and $g_1^{-1}g_2 \in P_{J(x_1)}$. Let $[g,x]$ denote the equivalence class of $(g,x)$.
\end{definition}

The group $G$ acts on $D(X, G(\mQ), \psi)$ by $g \cdot[g',x]=[gg', x]$. It is easy to see that $D(X, G(\mQ), \psi)$ has a structure of a polyhedral complex and that the $G$-action preserves that structure.
The stabilisers of this action are  the conjugates of local groups $P_J$ and the quotient is homeomorphic to $X\cong [e, X] \subset D(X, G(\mQ), \psi)$. Moreover $X \cong [e, X]$ is a so-called \emph{strict fundamental domain} for the $G$-action in the sense that it is a closed subset of $D(X, G(\mQ))$ intersecting every orbit in precisely one point.\medskip

In fact, any admissible action with a strict fundamental domain arises in the way described above. 

\begin{theorem}[{\cite[Proposition~II.12.20]{BH}}]\label{thm:scogisfund}
	Suppose a group $G$ acts admissibly on a connected polyhedral complex $X$ with a strict fundamental domain $Y \subset X$. 

	Then there is a strictly developable simple complex of groups $G(\mQ)$, where $\mQ$ is the poset of cells of $Y$ (ordered by the reverse inclusion) and where the local group at cell $e\subset Y$ is its $G$-stabiliser. The inclusion of cell stabilisers into $G$ defines a simple morphism $\psi \colon G(\mathcal Q) \to G$ such that $X$ is $G$-equivariantly homeomorphic to the Basic Construction $D(K, G(\mQ),\psi )$, where $K$ is the standard panel complex associated to $\mQ$. Moreover, if $X$ is simply connected then $G$ is isomorphic to the fundamental group of $G(\mQ).$ 
\end{theorem}

\begin{convention}
  In the case when $G$ is isomorphic to the fundamental group of $G(\mathcal Q)$ and the simple morphism $G(\mathcal Q) \to G$ is the canonical simple morphism $\iota$, we will omit the morphism from the notation and simply write  $D(X, G(\mQ))$ for the associated Basic Construction (where $X$ is a panel complex over $\mathcal Q$).
\end{convention}

\subsection{Thinning procedure}

\begin{definition}
  We say that a simple complex of groups $G(Q)$ is \emph{thin} if for any pair $J\leqslant T$ in $\mQ$, the monomorphism $\phi_{TJ} \colon P_J \to P_T$ is an isomorphism if and only if $J=T$.
\end{definition}

\begin{remark}
  Both in \cite{DMP} and \cite{PePry}, the assumption that a simple complex of groups is thin is a part of its definition.
\end{remark}

Below we describe a procedure of thinning, which given a strictly developable simple complex of groups $G(Q)$ results in a thin complex $G(\mR)$ together with a morphism of simple complexes of groups $G(\mQ) \to G(\mR)$ inducing an isomorphism of fundamental groups.

\begin{definition}[Block poset]
  Given a simple complex of groups $G(\mQ)$ with the collection of local groups $\{P_J\}_{J \in \mQ}$, let $\sim$ be an equivalence relation on $Q$ generated by 
    \[\begin{text}{$J\sim J'$  if $J \leqslant J' $ and $\phi_{J'J} \colon P_J \to P_{J'}$ is an isomorphism.}\end{text}\]

  An equivalence class $C$ of elements of $\mQ$ under relation $\sim$ is called a \emph{block}. There is a partial order on the set of blocks given by \[\text{$C\leqslant C'$ if and only if there exist $J\in C$ and $J'\in C'$ with $J\leqslant J'$.}\] Denote the associated poset by $\mR$ and call it the {\it block poset}.
\end{definition}

Note that there is a surjection of posets $\pi \colon \mQ \to \mR$ given by $J\in C \mapsto C$.

\begin{definition}[Thinning of a simple complex of groups]

  Let $G(\mQ)$ be a strictly developable simple complex of groups with the collection of local groups $\{P_J\}_{J \in \mQ}$ and the  fundamental group $G$.  Let $\mR$ be the block poset associated to $G(\mQ)$.

  Define a simple complex of groups $G(\mR)=(\{S_C\}_{C\in \mR}, \{ \psi_{C'C}\}_{C'\leqslant C \in \mR})$ as follows. For a block $C \in R$, let  $J \in \mQ$ be any element in the preimage $\pi^{-1}(C)$ and set $S_C=P_J$.
  Observe that $S_C$ is well-defined, since for all $J'\in \pi^{-1}(C)$ groups $P_{J'}$ are identified as a single subgroup of $G$. Now given two blocks $C \leqslant C'$ define the map \[\psi_{C'C} \colon S_C \to S_{C'}\] as the inclusion of the corresponding groups $P_J \leqslant P_{J'}$ seen as subgroups of $G$. Note that $G(\mR)$ is thin by construction.
\end{definition}

One easily verifies that $G(\mR)$ is strictly developable with  fundamental group isomorphic to $G$. Moreover, the surjection $\pi \colon \mQ \to \mR$ induces a morphism of simple complexes of groups $G(\mQ) \to G(\mR)$ which in turn induces an isomorphism on the fundamental groups (see \cite[Chapter~II.12]{BH} for a background on morphisms of simple complexes of groups). Finally, if $G(\mQ)$ is thin, then by definition $\mR$ is isomorphic to $\mQ$, and the morphism $G(\mQ) \to G(\mR)$ is an isomorphism.

\subsection{Bestvina complex}

\begin{definition}\label{def:upperlink}
  Let $(X, \{X_J\}_{J \in Q})$ be a panel complex over a poset $\mQ$. For an element $J \in \mathcal Q$ define the subcomplex $X_{>J}$ of $X$ by \[X_{>J}=\cup_{J < J'} X_{J'}.\]
\end{definition}

\begin{remark}
  In the case where $X=K$ is the standard panel complex over $\mQ$ we have \[K_{>J}= \abs{\{ J' \in \mQ \mid J' > J\}}.\]
\end{remark}

Observe that Theorem~\ref{thm:scogisfund} may be seen as evidence that the standard panel complex and the associated Basic Construction occur naturally. However, for computational purposes, a better suited panel complex is the following.

\begin{definition}[Bestvina complex]\label{def:genbestcx} 
  The {\it Bestvina panel complex} $(\B, \{\pan{\B}{J}\}_{J \in \mQ})$ is defined as follows. For every maximal element $J\in \mQ$, define $\pan{\B}{J}$ to be a point. Now given an element $J \in \mQ$ assume that for all $J'$ with $J < J'$ the panel $\pan{\B}{J'}$ has already been defined. Define $\pan{\B}{J}$ to be the compact contractible polyhedron  containing $\link{\B}{J}=\flink{\B}{J}$ of the smallest possible dimension.

  We define $B^{\Z}$ in the same way as $B$ except that we replace `contractible' by `acyclic' polyhedra. 
\end{definition}

\begin{remark} The panel complex $B$ was introduced by Bestvina in \cite{Best} for the poset of special subgroups of a finitely generated Coxeter group. It was extended to graph products of finite groups by Harlander and Meinert in  \cite{HarMei} and more generally to buildings that admit a chamber transitive action of a discrete group by Harlander  in \cite{Harland}.
\end{remark}

\begin{example}Consider finite groups $A,B,C$ with two inclusions $A \leqslant B$ and $A \leqslant C$. Consider two subgroups $E$ and $D$ of $C$, both containing the image of $A \leqslant C$. All inclusions are assumed to be proper.  Figure~\ref{fig:thinningandbestvina} depicts a complex of groups $G(\mQ)$ (where all the structure maps are the respective inclusions), its thinning $G(\mR)$ and the Bestvina complex associated to $\mR$. The fundamental group of $G(\mQ)$ (and hence of $G(\mR)$) is isomorphic to the amalgamated product $B\ast _{A}C$. Observe that poset $\mR$ has significantly fewer elements than $\mQ$. A further simplification is given by the Bestvina complex, whose dimension is lower than the dimension of $\abs{\mQ}$ and  $\abs{\mR}$. The Basic Construction $D(\B, G(\mathcal{R}))$ is isomorphic to the Bass-Serre tree of $B\ast _{A}C$.
\begin{figure}[!h]
  \centering
    \begin{tikzpicture}[scale=1]

      \definecolor{vlgray}{RGB}{230,230,230}
      \definecolor{red}{RGB}{100,150,000}

      \definecolor{tgreen}{RGB}{250,125,000}
      \definecolor{lred}{RGB}{150,200,050}

      \definecolor{blue}{RGB}{000,50,200}


      \begin{scope}

      \node [right]  at (-4,3.25)  {$a)$ $\mQ$ and $G(\mQ)$};

      \draw[white, pattern=custom north west lines,hatchspread=6pt,hatchthickness=1.5pt,hatchcolor=yellow!50] (-2,0)--(2,0)--(4,0)--(4,1)--(5,1.5)--(4,3)--(0,3)--(-2,0);

      \draw[red!30,line width=4] (-2,0) to (0,0);

      \draw[red!30,line width=4]  (-2,0) to (-1,1.5);

      \draw[red!30,line width=4]  (2,0) to (0,0);

      \draw[red!30,line width=4]  (2,0) to (1,1.5);

      \draw[red!30,line width=4]  (0,3) to (-1,1.5);
      \draw[red!30,line width=4]  (0,3) to (1,1.5);

      \draw[ ->-] (-2,0) to (0,0);
      \draw[->-] (-2,0) to (0,1);
      \draw[ ->-] (-2,0) to (-1,1.5);

      \draw[ ->-] (2,0) to (0,0);
      \draw[->-] (2,0) to (0,1);
      \draw[->-] (2,0) to (1,1.5);

      \draw[->-] (0,3) to (0,1);
      \draw[->-] (0,3) to (-1,1.5);
      \draw[->-] (0,3) to (1,1.5);

      \draw[->-] (0,0) to (0,1);
      \draw[->-] (-1,1.5) to (0,1);
      \draw[->-] (1,1.5) to (0,1);

      \draw[fill=red, red] (-2,0)   circle [radius=0.06];

      \draw[fill=red, red] (0,0)   circle [radius=0.06];

      \draw[fill=red, red] (2,0)   circle [radius=0.06];

      \draw[fill=tgreen, tgreen] (0,1)   circle [radius=0.06];

      \draw[fill=red, red] (-1,1.5)   circle [radius=0.06];

      \draw[fill=red, red] (1,1.5)   circle [radius=0.06];

      \draw[fill=red, red] (0,3)   circle [radius=0.06];

      \node [below ,red]   at (-2,0)  {$A$};

      \node [below ,red]   at (0,0)  {$A$};

      \node [below ,red]   at (2,0)  {$A$};

      \node [below left ,tgreen]   at (0.098,0.99)  {$B$};

      \node [below ,red]   at (-1,1.5)  {$A$};

      \node [below left,red]   at (1.1,1.5)   {$A$};

      \node [above ,red]   at (0,3)  {$A$};


      \begin{scope}[shift={(2,3)},rotate={180}]


      \draw[red!30,line width=4] (-2,0) to (0,0);
      \draw[red!30,line width=4] (-2,0) to (0,1);
      \draw[red!30,line width=4] (-2,0) to (-1,1.5);

      \draw[red!30,line width=4] (2,0) to (0,0);
      \draw[red!30,line width=4] (2,0) to (0,1);
      \draw[red!30,line width=4] (2,0) to (1,1.5);

      \draw[red!30,line width=4] (0,3) to (0,1);
      \draw[red!30,line width=4] (0,3) to (-1,1.5);
      \draw[red!30,line width=4] (0,3) to (1,1.5);

      \draw[red!30,line width=4] (0,0) to (0,1);
      \draw[red!30,line width=4] (-1,1.5) to (0,1);
      \draw[red!30,line width=4] (1,1.5) to (0,1);

      \draw[ ->-] (-2,0) to (0,0);
      \draw[->-] (-2,0) to (0,1);
      \draw[ ->-] (-2,0) to (-1,1.5);

      \draw[ ->-] (2,0) to (0,0);
      \draw[->-] (2,0) to (0,1);
      \draw[->-] (2,0) to (1,1.5);

      \draw[->-] (0,3) to (0,1);
      \draw[->-] (0,3) to (-1,1.5);
      \draw[->-] (0,3) to (1,1.5);

      \draw[->-] (0,0) to (0,1);
      \draw[->-] (-1,1.5) to (0,1);
      \draw[->-] (1,1.5) to (0,1);

      \draw[fill=red, red] (-2,0)   circle [radius=0.06];

      \draw[fill=red, red] (0,0)   circle [radius=0.06];

      \draw[fill=red, red] (2,0)   circle [radius=0.06];

      \draw[fill=red, red] (0,1)   circle [radius=0.06];

      \draw[fill=red, red] (-1,1.5)   circle [radius=0.06];

      \draw[fill=red, red] (1,1.5)   circle [radius=0.06];

      \draw[fill=red, red] (0,3)   circle [radius=0.06];

      \node [above ,red]   at (-2,0)  {$A$};
      \node [above ,red]   at (0,0)  {$A$};

      \node [above ,red]   at (2,0)  {$A$};

      \node [below left ,red]   at (-0.06 ,1.02)  {$A$};

      \node [below ,red]   at (-1,1.5)  {$A$};

      \node [below ,red]   at (0,3)  {$A$};


      \end{scope}

      \begin{scope}[shift={(4,0)}]


      \draw[ ->-] (-2,0) to (0,0);
      \draw[->-] (-2,0) to (0,1);
      \draw[ ->-] (-2,0) to (-1,1.5);

      \draw[ ->-] (2,0) to (0,0);
      \draw[->-] (2,0) to (0,1);
      \draw[->-] (2,0) to (1,1.5);

      \draw[->-] (0,3) to (0,1);
      \draw[->-] (0,3) to (-1,1.5);
      \draw[->-] (0,3) to (1,1.5);

      \draw[->-] (0,0) to (0,1);
      \draw[->-] (-1,1.5) to (0,1);
      \draw[->-] (1,1.5) to (0,1);

      \draw[fill=red, red] (-2,0)   circle [radius=0.06];

      \draw[fill=brown, brown] (0,0)   circle [radius=0.06];

      \draw[fill=red, red] (2,0)   circle [radius=0.06];

      \draw[fill=blue, blue] (0,1)   circle [radius=0.06];

      \draw[fill=red, red] (-1,1.5)   circle [radius=0.06];

      \draw[fill=violet,violet] (1,1.5)   circle [radius=0.06];

      \draw[fill=red, red] (0,3)   circle [radius=0.06];

      \node [below ,red]   at (2,0)  {$A$};

     \node [below left ,blue]   at (0.098,0.99)  {$C$};

      \node [below ,violet]   at (0.92,1.5)   {$E$};

      \node [below ,brown]   at  (0,0)  {$D$};


      \end{scope}

      \end{scope}


      \begin{scope}[shift={(0,-3)}]

      \node [right]  at (-4,1.5)  {$b) $ $\mR$ and $G(\mR)$};

      \begin{scope}[shift={(4,0)}]


      \draw[ ->-] (-2,0) to (0,0);
      \draw[->-] (-2,0) to (0,1);

      \draw[ ->-] (2,0) to (0,0);
      \draw[->-] (2,0) to (0,1);
      \draw[->-] (2,0) to (1,1.5);

      \draw[->-] (-2,0) to (-4,1);

      \draw[->-] (0,0) to (0,1);

      \draw[->-] (1,1.5) to (0,1);

      \draw[fill=red, red] (-2,0)   circle [radius=0.06];

      \draw[fill=brown, brown] (0,0)   circle [radius=0.06];

      \draw[fill=red, red] (2,0)   circle [radius=0.06];

      \draw[fill=blue, blue] (0,1)   circle [radius=0.06];

      \draw[fill=violet,violet] (1,1.5)   circle [radius=0.06];

      \draw[fill=tgreen, tgreen] (-4,1)   circle [radius=0.06];

      \node [below ,red]   at (2,0)  {$A$};
      \node [below ,red]   at (-2,0)  {$A$};

        \node [below left ,blue]   at (0.098,0.99)  {$C$};

      \node [below ,violet]   at (0.92,1.5)   {$E$};
      \node [below ,brown]   at  (0,0)  {$D$};

      \node [below left ,tgreen]   at (-3.92,0.99)  {$B$};


      \end{scope}

      \end{scope}


      \begin{scope}[shift={(0,-6)}]

      \node [right]  at (-4,1.75)  {$c) $ Bestvina complex for $\mR$};

      \begin{scope}[shift={(4,0)}]


      \draw[red!40,line width=4] (-4,1) to (0,1);


      \draw[fill=blue, blue] (0,1)   circle [radius=0.06];

      \draw[fill=tgreen, tgreen] (-4,1)   circle [radius=0.06];

      \node [below ,red]   at (-2,1)  {$A$};

      \node [below ,blue]   at (0,1)  {$C$};

      \node [below ,tgreen]   at (-4,1)  {$B$};

      \end{scope}

      \end{scope}

    \end{tikzpicture}
  \caption{Complex of groups $G(\mQ)$ together with its thinning $G(\mR)$ and the Bestvina complex associated to $\mR$. Elements of a block $[A]_1\subset \mQ$ with the local group $A$ are connected by green lines. The geometric realisation $\abs{\mQ_{\geqslant {[A]_1}}}$ is in yellow.}
  \label{fig:thinningandbestvina}
\end{figure}
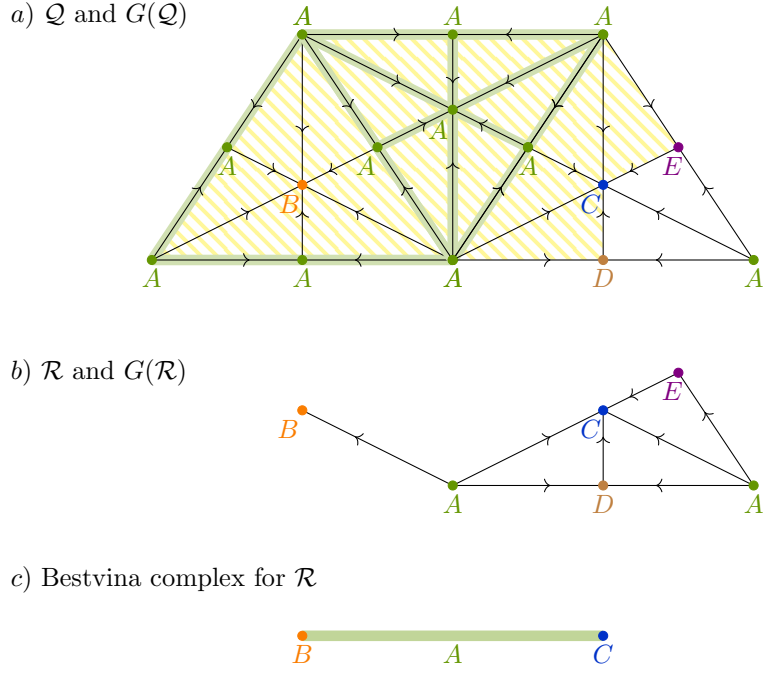 
\end{example}

The proof of the following proposition follows directly from Lemmas 2.4 and 2.5 of \cite{PePry}.
  
\begin{proposition}\label{prop:genhomotop}
  Let $G(\mathcal Q)$ be a strictly developable simple complex of groups and let $\psi \colon G(\mathcal Q) \to G$ be a simple morphism that is injective on local groups. Assume that $G(\mathcal Q)$ is thin. Then:
  \begin{enumerate}
  	\item \label{it:genhomotopspaces} the standard development $D(K, G(\mQ), \psi)$  and the Bestvina complex $D(\B, G(\mQ),\psi)$ are $G$-homotopy equivalent,
  	\item \label{it:genhomotopchains} The Bredon chain complexes ${C}^{\mF}_{\ast}(D(K, G(\mathcal{Q}),\psi))$ and ${C}^{\mF}_{\ast}(D(\B^{\Z}, G(\mathcal{Q}),\psi))$ are chain homotopy equivalent.
  \end{enumerate}
\end{proposition}

\begin{definition}[Local cohomological dimension]\label{def:dimbestv}
  For a poset $\mathcal Q$ define its \emph{local cohomological dimension} $\dB{Q}$ as follows
  \begin{equation*}
    \dB{Q}= \mathrm{max}\{n \in \mathbb{N} \mid \rH^{n-1}\big(\link{K}{J}\big) \neq 0 \text{ for some } J \in \mQ \}.
  \end{equation*}
\end{definition}

\begin{proposition}\label{prop:gendimbest} 
  We have the following equalities  
  \begin{align*}
                \dB{Q}  &= \mathrm{max}\{n \in \mathbb{N} \mid H^{n}(K_{J}, K_{> J}) \neq 0 \text{ for some } J \in \mQ \}\\
                        &= \mathrm{max}\{n \in \mathbb{N} \mid \rH^{n-1}(K_{> J}) \neq 0 \text{ for some } J \in \mQ \}\\
                        &= \mathrm{max}\{n \in \mathbb{N} \mid \rH^{n-1}(B_{> J}) \neq 0 \text{ for some } J \in \mQ \}\\
                        &=\dim (B^{\Z})\\
                        &=\mathrm{max}\{n \in \mathbb{N} \mid H^{n}(B^{\Z}_{J}, B^{\Z}_{> J}) \neq 0 \text{ for some } J \in \mQ \}.
  \end{align*}
  Moreover, 
  \[\mathrm{dim}(\B) = \left\{\begin{array}{lcr}

                                \dB{Q}          & \text{ if } & d \neq 2,\\
                                2 \text{ or } 3 & \text{ if } & d =2. \\

                              \end{array}\right.\]
\end{proposition}

\begin{proof} 
  The proof is essentially the same as the proof of \cite[Proposition 3.4]{PePry}.
\end{proof}

\begin{lemma}\label{lem:homot} 
  Let $G(\mathcal Q)$ be a strictly developable simple complex of groups with fundamental group $G$ and let $\mF$ be the family generated by local groups. Suppose $D(\D, G(\mathcal{Q}))$ is a model for $\EFG$. Let $\mR$ be the corresponding block poset. Then $ \fcd G\leqslant \dB{R}$. In particular, if $G(\mathcal Q)$ is thin then $ \fcd G\leqslant \dB{Q}$.
\end{lemma}

\begin{proof} 
  Consider the composition of chain maps $${C}^{\mF}_{\ast}(D(\D, G(\mathcal{Q})))\to {C}^{\mF}_{\ast}(D(T, G(\mathcal{R})))\to {C}^{\mF}_{\ast}(D(B^{\Z}, G(\mR))),$$
  where $K =\abs{\mQ}$, $T =\abs{\mR}$ and the complex $B^{\Z}$ is taken over poset $\mR$.

  The first map is induced by the map of Basic Constructions $D(\D, G(\mathcal{Q}))\to D(T, G(\mathcal{R}))$, which is in turn induced by a morphism of simple complexes of groups $G(\mQ) \to G(\mR)$. The second map is constructed in \cite[Theorem~A.1]{PePry} (it is straightforward to check that both the statement and the proof of Theorem~A.1 carry through for infinite local groups).

  Since $D(\D, G(\mathcal{Q}))$ is a model for $\EFG$, there is also a classifying $G$-map that  gives a chain map 
  $${C}^{\mF}_{\ast}(D(B^{\Z}, G(\mR)))\to {C}^{\mF}_{\ast}(D(\D, G(\mathcal{Q})))$$
  and the composition of both is chain homotopic to the identity on ${C}^{\mF}_{\ast}(D(\D, G(\mathcal{Q})))$. This shows that $ \fcd G\leqslant \dim (D(B^{\Z}, G(\mR)))= \dim (B^{\Z})$ and Proposition \ref{prop:gendimbest} finishes the proof.
\end{proof}

\section{Thin complexes of groups}\label{sec:minimal}

In this section, we will assume that the simple complex of groups $G(\mathcal Q)$ is thin. We will show that the Bredon cohomological dimension of the fundamental group of $G(\mathcal Q)$ is equal to the  local cohomological dimension of the poset $\mQ$.

\begin{proposition}\label{prop:chains} 

Let $G(\mathcal Q)$ be a thin simple complex of groups, let $G$ be a group, and let $\psi \colon G(\mathcal Q) \to G$ be a simple morphism which is injective on local groups. Suppose $(Y, \{Y_J\}_{J \in Q})$ is a simplicial panel complex over $\mQ$, and let $X=D(Y, G(\mathcal{Q}), \psi)$ be the associated Basic Construction. Then for any $J\in \mathcal Q$, there is an epimorphism of co-chain complexes
  $$\Psi: (C^*_{\mF}(X; \mA_{P_J}), \delta_{\mF}) \twoheadrightarrow (C^*(Y_{J}, Y_{>J}), \delta_J),$$
  where $P_J$ is a local group at $J \in \mQ$, seen as a subgroup of $G$.
\end{proposition}

\begin{proof} 
  Fix a dimension $i$, identify $Y$ with a subcomplex of $X$.  Let  $\sigma_j\isc Y$, $j=1, \dots, k$ be the $i$-simplices of $Y$ and denote by $G_{\sigma_j}\in \mF$ the stabiliser of $\sigma_j$. Then by the Yoneda Lemma, we obtain a natural equivalence 
  \begin{align*}
    C^i_{\mF}(X; \mA_{P_J})&= \nathom \big(\bigoplus_{j=1}^k \mA_{G_{\sigma_j}},\mA_{P_J}\big)\\
                            &\cong \bigoplus_{j=1}^k  \Z[\hom_G(G/{G_{\sigma_j}}, G/{P_J})].
  \end{align*}

  Given an $i$-simplex $\sigma\isc Y$ and a morphism  $\varphi:G/{G_{\sigma}}\xrightarrow{x} G/{P_J}: G_{\sigma}\mapsto xP_J$ in the summand indexed by $\sigma$, we define 
    \[\Psi(\varphi) = \left\{ \begin{array}{lll}
                                c^{\sigma} & \text{ if }\;  \sigma \isc Y_{J} \text{, }  G_{\sigma}=P_J   \text{ and } x\in P_J, & \text{ (type I)}\\
                                0          &  \text{otherwise }                                      & \text{ (type II)}
          \end{array} \right.\]
  where ${c^{\sigma}} \in C^i(Y_{J}, Y_{> J})$ equals to $1$ on $\sigma$ and vanishes everywhere else. 

	We claim that $\Psi$ is surjective. To see this, it is enough to note that if $\sigma \isc Y_J$ is an $i$-simplex with stabiliser $G_{\sigma}\gneq P_J$, then by the definition of Basic Construction we have that $\sigma \isc Y_{> J}$.
       
  It is left to check that $\Psi$ commutes with the co-boundary map. First, suppose $\varphi$ is of type II. Then $\delta_J(\Psi(\varphi))=0$. On the other hand, $\delta_{\mF}(\varphi)$ is a chain based at morphisms which are precomposed with $\varphi$ and hence of type II. To see this, suppose  $$\phi: G/{G_{\tau}}\xrightarrow{y^{-1}} G/{G_{\sigma}}\xrightarrow{x} G/{P_J}$$ is such a composition and it is of type I where   $\tau\isc Y$ is an $(i+1)$-simplex such that $y\tau$ contains $\sigma$ as a face.  Since $Y$ is a strict fundamental domain, observe that $y \in G_{\sigma}$. 

  Since $\phi$ is of type I, we must have  $G_{\tau}=P_J$ and $y^{-1}x\in P_J$, which implies that $G_{\sigma}=P_J$. Since now $x\in P_J$, this shows that $\varphi$ is of type I, which is a contradiction.  Therefore, $\Psi(\delta_{\mF}(\varphi))=0$.
        
  Now, suppose $\varphi$ is of type I, i.e.,\ $\varphi=\varphi_{\sigma}:G/{P_J}\xrightarrow{1} G/P_J$ with $P_J$ the stabiliser of $\sigma$. Then 
  \begin{equation}\label{gfirst}
    \delta_J(\Psi(\varphi_{\sigma}))=\delta_J(c^{\sigma}) =\sum_{t=1}^l(-1)^{\text{sgn}(\tau_t)}c^{\tau_t}
  \end{equation}
where $\tau_t\isc Y_J$ contains $\sigma$ as a face. On the other hand,
  \begin{equation}\label{gsecond} 
    \Psi(\delta_{\mF}(\varphi_{\sigma}))=\sum_{s=1}^r(-1)^{\text{sgn}(y_s\tau_s)}\Psi(\varphi_{y_s\tau_s})
  \end{equation} 
  where $y_s\in G$, $\tau_s\isc Y$, $y_s\tau_s$ is an $(i+1)$-simplex containing $\sigma$ as a face and $\varphi_{y_s\tau_s}:G/{G_{\tau_s}}\xrightarrow{y_s^{-1}} G/{P_J}$. Since $Y$ is a strict fundamental domain, $y_s^{-1}\sigma=\sigma$ and hence $y_s\in P_J$. Note that if $\Psi(\varphi_{y_s\tau_s})\ne 0$, then by definition of $\Psi$, we have  $y_s\tau_s\subset Y_J$ and $G_{\tau_s}=P_J$. Therefore, $\tau_s=y_s\tau_s\isc Y_J$ and $\varphi_{y_s\tau_s}=\varphi_{\tau_s}$. In this case, we have $\Psi(\varphi_{\tau_s})=c^{\tau_s}$.
  The claim now follows from equating (\ref{gfirst}) and (\ref{gsecond}).
\end{proof}

\begin{proposition}\label{prop:partial} 
  If  $D(\D, G(\mathcal{Q}), \psi)$ is a model for $\EFG$, then $\fcd G=\dB{Q}$.
\end{proposition}

\begin{proof}
  Note that by the assumption $D(\D, G(\mathcal{Q}), \psi)$ is simply connected, and thus by Theorem~\ref{thm:scogisfund} we have that $G$ is necessarily isomorphic to the fundamental group of $G(\mathcal{Q})$. Consider the panel complex $\B^{\Z}$ given in Definition~\ref{def:genbestcx}. By passing to a barycentric subdivision we can assume that $\B^{\Z}$ is a simplicial panel complex. Let $X=D(\B^{\Z}, G(\mathcal{Q}), \psi) $. By Proposition~\ref{prop:genhomotop}(\ref{it:genhomotopchains}) we have that ${C}^{\mF}_{\ast}(D(K, G(\mathcal{Q}),\psi))$ and ${C}^{\mF}_{\ast}(X)$ are chain homotopy equivalent and thus the  latter can be used to compute $\mathrm{H}^n_{\mathcal{F}}(G,-)$.
  
	Now Proposition \ref{prop:gendimbest} implies that there exists $J\in \mathcal Q$ such that $$H^{\dB{Q}}(B^{\Z}_{J}, B^{\Z}_{> J}) \neq 0.$$ Since $C^i_{\mF}(X; \mA_{P_J})=0$ for $i>\dB{Q}$, by Proposition \ref{prop:chains},  $\Psi$ induces an epimorphism $$\Psi^*:H^{\dB{Q}}_{\mF}(X; \mA_{P_J})\to H^{\dB{Q}}(B^{\Z}_{J}, B^{\Z}_{> J}).$$ This shows that $H^{\dB{Q}
  }_{\mF}(X; \mA_{P_J})\neq 0$ and hence, by Lemma \ref{lem:homot}, we obtain $\fcd G=\dB{Q}$.
\end{proof}

\section{Cohomology of simple complexes of groups}\label{sec:cohofscofgs}   

Let $G(\mathcal Q)$ be a simple complex of groups and let $\psi \colon G(\mathcal Q) \to G$ be a simple morphism which is injective on local groups.  Recall that by Convention~\ref{conv:localgroupsaresubgroups}, for any $J \in \mQ$ we identify $P_J$ with $\psi(P_J) \leqslant G$.

 For $ J \in \mQ$ let $\mathcal I_J$ be a complete set of representatives of the set $$\{g\in G \;|\; g^{-1}P_J g = P_U \mbox{ for some  } U \in \mQ \}/{P_J},$$ where $P_J$ acts by left multiplication.

Suppose $\Omega \subseteq \mQ$ is a subset such that $P_U=P_{U'}$ for all $U, U'\in \Omega$. Define  subcomplexes $K_{\Omega}$ and $K_{>\Omega}$ of $K$ to be 
\begin{align*}
  K_{\Omega} & = \abs{ \{ V \in \mathcal Q \mid V \geqslant U \mbox{ for some } U \in \Omega\}},\\
 K_{>\Omega} & = \abs{ \{ V \in \mathcal Q \mid V \geqslant U \mbox{ for some } U \in \Omega   \text{ and } P_V\gneq P_U\}}.
\end{align*}

For $J \in \mQ$ and $g\in G$ define $$\Omega^g_J=\{U\in \mathcal Q \;|\; P_U=g^{-1}P_Jg\}.$$

\begin{proposition}\label{prop:ggchains} 
  Suppose that $G(\mathcal Q)$ is a strictly developable simple complex of groups and let $\psi \colon G(\mathcal Q) \to G$ be a simple morphism which is injective on local groups. Let $X=D(\D, G(\mathcal{Q}), \psi)$ be the associated Basic Construction. Then for any $J\in \mathcal Q$, there is an isomorphism of co-chain complexes$$\Phi: (C^*_{\mF}(X; \mB_{P_J}), \delta_{\mF}) \to \bigoplus_{g\in \mathcal I_J} (C^*(K_{\Omega^g_J}, K_{>\Omega^g_J}), \delta).$$
\end{proposition}

\begin{proof}  
  We define $\Phi=\bigoplus_{g\in \mathcal I_J} \Psi_{g}$  with each 
  $$\Psi_{g}: C^*_{\mF}(X; \mB_{P_J}) \to C^*(K_{\Omega^g_J}, K_{>\Omega^g_J})$$ 
  constructed analogously to the map $\Psi$ of Proposition \ref{prop:chains} where one replaces an arbitrary simplicial panel complex $Y$ with $K$. Namely, we identify
  \begin{align*}
    C^i_{\mF}(X; \mB_{P_J}) & = \nathom \big(\bigoplus_{\sigma \isc K^{(i)}}  \mA_{G_{\sigma}},\mB_{P_J}\big),\\
                            & \cong \bigoplus_{\sigma \isc K^{(i)}} \mB_{P_J}(G_{\sigma}),\\
                            &\cong \bigoplus_{j=1}^k  \Z[\mbox{isom}_G(G/{G_{\sigma_j}}, G/{P_J})],
                                                                   \end{align*}
  \noindent where $\sigma_j$-s are all the $i$-simplices of $K$ so that $G_{\sigma_j}=_G P_J$.

  Now, fix $g\in \mathcal I_J$ and  suppose $\sigma$ is an $i$-simplex with stabiliser $g^{-1}P_Jg=P^g_J$. Given an (iso)morphism 
  $$(\varphi:G/{G_{\sigma}}\xrightarrow{x} G/{P_J}: G_{\sigma}\mapsto xP_J)\in C^i_{\mF}(X; \mB_{P_J}),$$
   we define 
  \[\Psi_{g}(\varphi) = \left\{ \begin{array}{ll}
        c^{\sigma}      & \text{ if }\;   G_{\sigma}=P^g_J   \text{ and } x\in g^{-1}P_J\\
        0 &  \text{otherwise }   
        \end{array} \right.\]
        
        \noindent where ${c^{\sigma}}\in C^i(K_{\Omega^g_J}, K_{>\Omega^g_J})$ equals to $1$ on $\sigma\isc K_{\Omega^g_J}$ and  vanishes everywhere else. The proof that $\Psi_{g}$ commutes with the co-boundary maps is analogous to the corresponding argument in the proof of Proposition \ref{prop:chains} and hence it is omitted.
    (Alternatively, it also follows from the commutativity of the co-boundary maps with sections $\Delta_{g}$ defined below.)

  To show that $\Phi$ is an isomorphism, we first define a section 
  $$\Delta_{g}: C^*(K_{\Omega^g_J}, K_{>\Omega^g_J})\to C^*_{\mF}(X; \mB_{P_J}): {c^{\sigma}}\mapsto (\varphi_{\sigma}: G/{P^g_J}\xrightarrow{g^{-1}} G/P_J)$$
   to each $\Psi_{g}$. We need to show that it commutes with the co-boundary maps. We have
  \begin{equation}\label{gf}
    \delta_{\mF}(\Delta_{g}(c^{\sigma}))=\delta_{\mF}(\varphi_{\sigma})=\sum_{s=1}^r(-1)^{\text{sgn}(y_s\tau_s)}\varphi_{y_s\tau_s}
  \end{equation}
  where $\tau_s\isc K$ contains $\sigma$ as a face and $\varphi_{y_s\tau_s}:G/{G_{\tau_s}}\xrightarrow{y_s^{-1}} G/{P^g_J}\xrightarrow{g^{-1}} G/P_J$ with $y_sG_{\tau_s}y_s^{-1} \leqslant P^g_J$ and  $y_s\in P^g_J$. Note that if $0\ne \varphi_{y_s\tau_s}\in C^*_{\mF}(X; \mB_{P_J})$, then by definition of $\mB_{P_J}$, the subgroup $G_{\tau_s}$ must be conjugate to $P_J$ and  $G_{\tau_s}=P^g_J$. Therefore, $\tau_s=y_s\tau_s\isc K$ and $\varphi_{y_s\tau_s}=\varphi_{\tau_s}\in \mbox{Im}\Delta_g$. 
         
  On the other hand,
  \begin{equation}\label{gs}
    \Delta_{g}(\delta_J(c^{\sigma}))=\sum_{t=1}^l(-1)^{\text{sgn}(\tau_t)}\Delta_{g}(c^{\tau_t})=\sum_{t=1}^l(-1)^{\text{sgn}(\tau_t)}\varphi_{\tau_t}  \end{equation}
	where $\tau_t\isc K_{\Omega^g_J}$ contains $\sigma$ as a face.  The claim now follows from equating (\ref{gf}) and (\ref{gs}). It is straightforward to check that $\Phi$ and $\Delta=\bigoplus_{g\in \mathcal I_J} \Delta_{g}$ are inverses of each other.
\end{proof}

\begin{remark}\label{rem:otherpanelcomplex}
	It is worth pointing out that Proposition~\ref{prop:ggchains} can be generalised to hold for  arbitrary  simplicial  panel complex $(X, \{X_J\}_{J \in Q})$ where one defines $X_{\Omega}= \bigcup_{J \in \Omega} X_J$ and $X_{>\Omega} = \bigcup_{ \{ U \in \mQ \mid  U\geq J\text{ for some } J \in \Omega \text{ and }  P_U \gneq P_J \} }X_U$, though this is not necessary for our purposes.
\end{remark}

\section{Main theorems}\label{sec:maintheorems}

In this section we state and prove slightly more general versions of Theorems~\ref{thm:introrigidgen} and ~\ref{thm:intromain} from the introduction. The generalisation concerns the computation of Bredon cohomology of the Basic Construction  $D(\D, G(\mathcal{Q}), \psi)$. In the statements below, we allow $\psi \colon G(\mQ) \to G$ to be a simple morphism to an arbitrary group $G$, not necessarily the fundamental group of $G(\mQ)$.

\begin{theorem}\label{thm:textrigidgen}

 Let $G(\mathcal Q)$ be a strictly developable simple complex of groups and let $\psi \colon G(\mQ) \to G$ be a simple morphism that is injective on local groups. Let $\mF$ be the family of subgroups of $G$ generated by local groups. Let $X=D(\D, G(\mathcal{Q}), \psi)$ be the associated Basic Construction. For $J\in \mQ$ we then have
  \begin{equation}\label{eq:textrigidzero}
    H^*_{\mF}(X; \mB_{P_J}) \cong \bigoplus_{g\in \mathcal I_J}\bigoplus_{C_J^g\subseteq \Omega^g_J} H^*(K_{C_J^g}, K_{>C_J^g}),
  \end{equation}
   where $C_J^g\subseteq \Omega^g_J$ denotes a block in $\Omega^g_J$.\medskip

  If $G(\mathcal Q)$ is rigid and $X$ is a model for $\EFG$, then 
  \begin{equation}\label{eq:textrigid}
    \fcd G=\mathrm{max}\{n \in \mathbb{N} \mid H^{n}(K_C, K_{> C}) \neq 0 \text{ for some block } C \subseteq {\mathcal Q}\}.
  \end{equation} 
\end{theorem}

\begin{proof}
  First we prove formula~\eqref{eq:textrigidzero}. To do this we show that for every $J \in \mathcal Q$, ${g\in \mathcal I_J}$ and for any integer $n\geqslant 0$ we have 
  \begin{equation}\label{eq:dimensioncompare}
    H^n(K_{\Omega^g_J},K_{>\Omega^g_J})\cong \bigoplus_{C_J^g\subseteq \Omega^g_J} H^n(K_{C_J^g}, K_{>C_J^g}).
  \end{equation}
  To show \eqref{eq:dimensioncompare}, we proceed by induction on the number of blocks $C\subseteq \Omega^g_J$. If $\Omega^g_J$ contains only one block then \eqref{eq:dimensioncompare} is clearly satisfied. Assume now that $\Omega^g_J$ contains more than one block. Let $C \subseteq \Omega^g_J$, let $R= \Omega^g_J \smallsetminus C$ and write the pair $(K_{\Omega^g_J}, K_{> \Omega^g_J})$ as
  \[(K_{\Omega^g_J}, K_{> \Omega^g_J})= (K_{R} \cup K_{C}, K_{> R} \cup K_{>C}).\]
  Consider the relative Mayer-Vietoris sequence for the above pair:
  \begin{align*} 
            H^{n-1}(K_{C}  \cap K_{R},K_{> C} \cap K_{> R} )  \to & \\ 
        \to H^n(K_{\Omega^g_J}, K_{> \Omega^g_J})\to H^n  (K_{C}&, K_{> C}) \oplus  H^n(K_{R}, K_{> R})  \to \\
                                                                    & \to H^{n}(K_{C} \cap K_{R},K_{> C} \cap K_{> R} ).  
  \end{align*} 
  \begin{claim}
    We have $K_{C}  \cap K_{R} =K_{> C}  \cap K_{> R}$.
  \end{claim}

  To prove the claim consider an element $V \in K_{C}  \cap K_{R}$ (i.e., we view $V\in \mathcal Q$ as a vertex of $K$). Thus $U \leqslant V$ and $U' \leqslant V$ for some $U\in C$ and $U' \in R$. If $V\notin K_{> C}  \cap K_{> R}$, then $P_V=P_U$ or $P_V=P_{U'}$. In either case we get $P_V=g^{-1}P_Jg$ which implies that $V\in C$ and $V\in R$. This is a contradiction and the claim follows.

  The claim implies that $H^{n}(K_{C}  \cap K_{R},K_{> C} \cap K_{> R} )=0$ for every $n\geqslant 0$ and therefore the map
  \[H^n(K_{\Omega^g_J}, K_{> \Omega^g_J}) \to H^n(K_{C}, K_{> C}) \oplus  H^n(K_{R}, K_{> R})\] 
  is an isomorphism. Since by the inductive assumption we have 
  $$H^n(K_{R},K_{>R})\cong \bigoplus_{C' \subseteq R}H^n(K_{C'},K_{>C'}),$$
  the formula \eqref{eq:dimensioncompare} is established.

  Formula~\eqref{eq:textrigidzero} follows now easily from Proposition~\ref{prop:ggchains} and formula~\eqref{eq:dimensioncompare}.\bigskip

  We now prove formula~\eqref{eq:textrigid}. Note that here by the assumption $X$ is a cocompact model for $\EFG$ and thus $G$ is isomorphic to the fundamental group of $G(\mQ)$ (see Theorem~\ref{thm:scogisfund}). By Corollary \ref{cor:B_module}, we have
  $$\fcd G=\mathrm{max}\{n \in \mathbb{N} \mid H^{n}_{\mF}(X, \mB_{P_J}) \neq 0  \text{ for some } J \in \mathcal Q \}.$$
  By Proposition \ref{prop:ggchains}, we have
  \begin{align*}
      &\mathrm{max}\{n \in \mathbb{N} \mid H^{n}_{\mF}(X, \mB_{P_J}) \neq 0  \text{ for some } J \in \mathcal Q \}\\
    = &\mathrm{max}\{n \in \mathbb{N} \mid H^{n}(K_{\Omega^g_J}, K_{> \Omega^g_J}) \neq 0  \text{ for some } J \in \mathcal Q, g\in \mathcal I_J \}\\
    = &\mathrm{max}\{n \in \mathbb{N} \mid H^{n}(K_{\Omega^1_U}, K_{> \Omega^1_U}) \neq 0  \text{ for some } U \in \mathcal Q \}\\
    = &\mathrm{max}\{n \in \mathbb{N} \mid H^{n}(K_{C}, K_{> C}) \neq 0  \text{ for some block } C \subseteq \mathcal Q \}\qedhere
  \end{align*}
\end{proof}

\begin{proof}[Proof of Theorem \ref{thm:intromain}]
  We first prove part~\ref{it:HoeqPlusFormulaCohomology}. By Proposition~\ref{prop:genhomotop}(\ref{it:genhomotopspaces}) complexes $D(K, G(\mQ))$ and $D(\B, G(\mQ))$ are $G$-homotopy equivalent. 

  The formula for cohomology of $D(K, G(\mQ))$ follows from formula~\eqref{eq:textrigidzero} of Theorem~\ref{thm:textrigidgen} in the following way (note that in formula~\eqref{eq:textrigidzero} one does not assume rigidity). Since by assumption complex $G(\mQ)$ is thin, we have that blocks are equal to elements of $\mQ$. Moreover, for a single element $U \in \mQ$ we have that $K_U$ is contractible, and thus we obtain \[H^*(K_{U}, K_{>U}) \cong \rH^{*-1}(K_{>U}).\]

  Now we prove part~\ref{it:mainintroCohDim}. Since $D(K, G(\mQ))$ is a model for $\EFG$, it is in particular simply connected, and thus by Theorem~\ref{thm:scogisfund} we get that $G$ is isomorphic to the fundamental group of $G(\mQ)$.  Since $D(K, G(\mQ))$ and $D(\B, G(\mQ))$ are $G$-homotopy equivalent, we conclude that $D(\B, G(\mQ))$  is a model for $\EFG$ as well. Clearly $D(\B, G(\mQ))$ is cocompact. 

  The formula for the dimension of  $D(\B, G(\mQ))$ and formula \eqref{eq:introdim} for $\fcd G$ follow now easily from combining Proposition~\ref{prop:gendimbest} and Proposition~\ref{prop:partial}. 
\end{proof}

\begin{remark}We remark that Theorem~\ref{thm:textrigidgen} holds true if we replace the complex $K$ by any other panel complex over $\mQ$ whose all panels are contractible (cf.\ Remark~\ref{rem:otherpanelcomplex}). In particular, one can use the Bestvina complex $B$. Unlike in Theorem~\ref{thm:intromain}, here the dimension of the  resulting Basic Construction $D(\B, G(\mQ), \psi)$ may not be optimal, nonetheless, since Bestvina complex in general has a smaller cell structure than the complex $K$, its use may simplify cohomological computations.	
\end{remark}

\section{Deformation retractions and actions on trees}\label{sec:defretraction}

  In this section we show that when Bestvina complex for $G(\mathcal Q)$ is a tree then it  can be realised as an equivariant deformation retract of the standard development. This can be seen as a generalisation of results of Davis \cite[Proposition 8.5.5]{Davbook} and the authors' \cite{PePry} to the case of infinite local groups. The key ingredient in the proof is the cohomological formula of Theorem~\ref{thm:intromain}. We remark that our approach relies neither on Dunwoody's accessibility theory \cite{Dun} nor on Dicks-Dunwoody's almost stability theorem \cite[III.8.5]{DicDun}.

  \begin{theorem}\label{thm:deformationretraction}
    Let $G(\mathcal Q)$ be a strictly developable thin simple complex of groups over a poset $\mathcal{Q}$ with fundamental group $G$ and let $\mF$ be the family generated by local groups. Suppose that $D(\D, G(\mathcal{Q}))$ is a model for $\EFG$. Then $ \fcd G\leqslant 1$ if and only if $D(B, G(\mathcal{Q}))$ is a tree and an equivariant deformation retract of  $D(K, G(\mathcal{Q}))$.
  \end{theorem}

  \begin{proof}
    The proof is a verbatim translation of the proof of Theorem~4.8 of \cite{PePry} which treats the case of finite local groups. The only place where that proof uses the fact that local groups are finite is the use of \cite[Proposition~3.6]{PePry}, which gives a formula for the cohomological dimension of $G$ for the family of finite subgroups. In Theorem~\ref{thm:intromain} we prove that the same formula holds for a family $\mF$ generated by arbitrary local groups:
    $$\fcd G=\mathrm{max}\{n \in \mathbb{N} \mid \rH^{n-1}(K_{> J}) \neq 0 \text{ for some } J \in \mathcal Q \}.$$
    Note that $\fcd G \leqslant 1$ implies that for any $J \in \mathcal Q$ we have $\rH^{n}(K_{> J}) = 0$ for all $n>0$, and thus any $K_{> J}$ is a disjoint union of contractible spaces. This is the crucial piece of geometric information which is used in Theorem~4.8 of \cite{PePry} to build Bestvina complex as an equivariant deformation retract of the standard development.
  \end{proof}

  In some cases the condition ensuring that $\fcd G\leqslant 1$ can be read from the global structure of the poset $\mathcal Q$.

  \begin{example}\label{ex:chordal}
    Suppose $Q$ is a poset of simplices of a finite flag simplicial complex $L$. Then $\dB{Q} \leqslant 1$ if and only if the one skeleton $L^{(1)}$ of $L$ is a \emph{chordal graph}, i.e., for any cycle in $L^{(1)}$ of length at least four there is an edge connecting two non-consecutive vertices of the cycle (a \emph{chord}).
  \end{example}

\section{Applications and examples}\label{sec:applications}

\subsection{Bredon cohomological dimension for finite subgroups}

\begin{proposition}\label{prop:bredondimforfiniteandtorsionfree}
  Let $G(\mathcal Q)$ be a strictly developable simple complex of groups $G(\mathcal Q)$ with collection of local groups $\{P_J\}_{J \in \mathcal Q}$ and  fundamental group $G$. Let $\mR$ be the associated block poset. Suppose $D(\D, G(\mathcal{Q}))$ is a model for $\EFG$ where $\mathcal F$ is the family generated by local groups and assume that $\mathcal F$ contains all finite subgroups of $G$. Then 
  \[\ucd G \leqslant \dB{R}+ \mathrm{max}  \{  \ucd P_J \mid J \in \mathcal Q  \}.\]
  In particular, if $G(\mathcal Q)$ is thin then 
  \[\ucd G \leqslant \dB{Q}+ \mathrm{max}  \{  \ucd P_J \mid J \in \mathcal Q  \}.\]

  If $G$ is virtually torsion-free then both inequalities remain true if one replaces `$\ucd\! $' by `$\vcd\! $'.
\end{proposition}

\begin{proof}
  For any discrete group $G$ and for any family of subgroups $\mF$ which contains all finite subgroups of $G$ we have $\ucd G \leqslant \fcd G +  \mathrm{max}  \{  \ucd F \mid F \in \mathcal \mF  \} $ \cite[Corollary~4.2]{DPT}. Since every subgroup in $\mF$ is subconjugate to a subgroup in $\{P_J\}_{J \in \mathcal Q}$, we get that $\mathrm{max}  \{  \ucd F \mid F \in \mathcal \mF  \} =\mathrm{max}  \{  \ucd P_J \mid P_J \in \mathcal \mQ  \} $. Both claims now follow from Lemma~\ref{lem:homot}.

  For the virtually torsion-free case, one first replaces $G$ with a torsion-free finite-index subgroup $G'$ and then one performs the same argument as above applied to ordinary cohomological dimension instead of the proper cohomological dimension. 
\end{proof}

\subsection{Cohomology of buildings and their automorphisms}\label{subsec:autbuildings}

 Groups acting chamber transitively on buildings  form a large class of examples of actions  on nonpositively curved  complexes with a strict fundamental domain.\medskip

We recall some terminology. Let $(W, S)$ be a Coxeter system with the set $S$ finite. A subset $J\subseteq S$ is called \emph{spherical} if the elements of $S$ generate a finite subgroup of $W$ (we assume that the empty set $\emptyset \subseteq S$ generates the trivial subgroup and thus it is spherical). Let $\mQ$ be the poset of spherical subsets of $S$ ordered by inclusion.

Now suppose that $\Delta$ is a building of type $(W,S)$ and that a group $G$ acts chamber transitively on $\Delta$ (see \cite[Section~I.3]{Davbuild}). Such an action gives rise to a strictly developable simple complex of groups $G(\mQ)$ with fundamental group isomorphic to $G$. The standard \emph{geometric realisation} of $\Delta$ is by definition the Basic Construction $D(\D, G(\mathcal{Q}))$ (by replacing $K$ with another panel complex over $\mQ$ one obtains a variety of geometric realisations of $\Delta$). By \cite[Theorem~11.1]{Davbuild} there is a complete $\mathrm{CAT}(0)$ metric on $D(\D, G(\mathcal{Q}))$ such that $G$ acts by isometries. Thus $D(\D, G(\mathcal{Q}))$ is a cocompact model for $\EFG$, where $\mF$ is the family generated by the local groups. Let $P_J$ denote the local group (i.e.~ the special parabolic subgroup) at element $J \in Q$. 

\begin{lemma}\label{lem:buildingmin} 
  In the above setting, if $J \leqslant T$ then $P_{J} \leqslant P_T$ is a proper inclusion. In particular, the complex of groups $G(\mathcal Q)$ is thin. 
\end{lemma}

\begin{proof}
  The proof is verbatim the same as the proof of \cite[Lemma~5.1]{LePe}, since the assumption that $G$ acts properly on $\Delta$ was not used there.
\end{proof}

We remark that in the case where $G=W$, the standard geometric realisation $D(\D, G(\mathcal{Q}))$ of $\Delta$ is by definition the Davis complex of the system $(W,S)$ and it is denoted by $\Sigma_W$.\smallskip

We are now ready to prove the main result of this section which is Corollary \ref{cor:introbuilding}

\begin{proof}[Proof of Corollary \ref{cor:introbuilding}]
  By definition $D(\B, G(\mathcal{Q}))$ is a realisation of $\Delta$. Since by Lemma~\ref{lem:buildingmin} complex  $G(\mQ)$ is thin, Proposition \ref{prop:genhomotop}(\ref{it:genhomotopspaces}) implies that $D(\B, G(\mathcal{Q}))$ and $D(\D, G(\mathcal{Q}))$ are $G$-homotopy equivalent. Thus  $D(\B, G(\mathcal{Q}))$  is a model for $\EFG$, since  $D(\D, G(\mathcal{Q}))$ is a model. Since $D(\B, G(\mathcal{Q}))$ is clearly cocompact, this establishes the first claim of the theorem.
  
  The remaining claims follow directly from Theorem~\ref{thm:intromain} as the formula  for $\vcd W$ (see \cite[Theorem~2]{Drani} or \cite[Theorem~5.4]{DMP}) is identical to formula \eqref{eq:introdim} for $\fcd G$.
\end{proof}

\begin{remark}  
  We note that $D(\B, G(\mathcal{Q}))$ can also be constructed by first constructing $D(\B, W(\mathcal{Q}))$ for the corresponding Coxeter group $W$ and then realising the building with apartments modelled on $D(\B, W(\mathcal{Q}))$. 
\end{remark}

We obtain the following corollary, first proven in \cite[Theorem 4.1.(ii)]{Harland}.

\begin{corollary}\label{cor:textbuilding_3} 
  Let $G$ be a virtually torsion-free group acting chamber transitively on a building of type $(W,S)$. Then $$\vcd G\leqslant \vcd W + \max \{\vcd P \;|\; P \mbox{ is a special parabolic subgroup of  } G\}.$$
\end{corollary}

\begin{proof}
  The corollary follows easily from combining Corollary~\ref{cor:introbuilding} with Proposition~\ref{prop:bredondimforfiniteandtorsionfree}, and the facts that $\dB{Q}=\vcd W$ and that local groups of $G(\mathcal Q)$ are precisely the special parabolic subgroups of $G$.
\end{proof}

\subsection{Graph products of groups} An example of a group acting chamber transitively on a building is a \emph{graph product} of groups, such as for example the right-angled Artin group or the right-angled Coxeter group. 

\begin{definition}
  Consider a finite flag simplicial complex $L$ on the vertex set $S$ with groups $P_s$ for every $s \in S$. The \emph{graph product} $G_L$ is defined as the quotient of the free product of groups $P_s$ for $s \in S$ by the relations 
  \begin{equation*}
    \text{$\{ [P_s,P_t]$ if $[s,t]$ is an edge of $L \}$}.
  \end{equation*}
  In other words, elements of subgroups $P_s$ and $P_t$ commute if and only if there is an edge $[s,t]$ in $L$. 
\end{definition}

If we set $P_s \cong \mathbb{Z}/2$ for every $s \in S$, the corresponding graph product is called the \emph{right-angled Coxeter group} and it is denoted by $W_L$.

If we set $P_s \cong \mathbb{Z}$ for every $s \in S$, the corresponding graph product is called the \emph{right-angled Artin group} and it is denoted by $A_L$.

\begin{theorem}\cite[Theorem~5.1]{Davbuild}\label{thm:davgraphroduct}
  The group $G_L$ acts chamber transitively on a building  of type $(W_L,S)$, where $W_L$ is the right-angled Coxeter group corresponding to $L$. 
\end{theorem}

Thus $G_L$ is the fundamental group of a simple complex of groups $G(\mQ)$, where $\mQ$ is the poset of spherical subsets of $S$. Note that $\mQ$ can be identified with the poset of simplices of $L$ ordered by inclusion, together with the smallest element corresponding to the empty set. Consequently, the geometric realisation of $\mQ$ is isomorphic to the cone over the barycentric subdivision of $L$. Moreover, the local group at simplex $\sigma$ of $L$ is the direct product $\prod_{s \in \sigma}P_{s}$ and the local group at $\emptyset$ is the trivial group.\medskip

Theorem~\ref{thm:davgraphroduct} implies that Corollaries~\ref{cor:introbuilding}, \ref{cor:textbuilding_3}, \ref{cor:introdeformationretraction} and \ref{cor:introbuilding_2} apply to $G_L$.

\subsection{Examples}

\begin{example}[Barycentric subdivision and thinning]\label{ex:thick-complex}

The first example shows that the thinning procedure may be intuitively seen as an inverse to the barycentric subdivision.

  Let $X$ be a $G$-simplicial complex with a strict fundamental domain $Y$, let $G(\mathcal Q)$ be the associated complex of groups and let $\mF$ be the family generated by the stabiliser subgroups. Thus $\mQ$ is the poset of simplices of $Y$ (ordered by the reverse inclusion). Assume that $G(\mQ)$ is thin.

  Now let $ X'$ denote the barycentric subdivision of $X$, and consider the induced action of $G $ on  $X'$. The fundamental domain for this action is clearly  $ Y'$.  Let $G(\mathcal Q')$ be the associated simple complex of groups, where  $Q'$ is the poset of simplices of $Y'$. Observe that $G(\mQ')$ is not thin.
  
  One easily sees that the fundamental groups of $G(\mathcal Q)$ and  $G(\mathcal Q')$ are isomorphic, and so are the families generated by local groups. However, we have \[\dB{Q'}=\mathrm{dim}(X')=\mathrm{dim}(X),\] while in general $\dB{Q}$ is strictly less than $\mathrm{dim}(X)$. 

\begin{proposition}
  Let $G(\mQ)$ and $G(\mQ')$ be as above. Let $\mR$ denote the block poset associated to  $G(\mQ')$  and let $G(\mR)$ be the thinning of $G(\mQ')$. Then $\mQ$ and  $\mR$ are isomorphic, and simple complexes of groups $G(\mQ)$ and $G(\mR)$ are simply isomorphic. 
\end{proposition}

\begin{proof}Given a simplex $\sigma \subset Y$, all simplices of $Y'$ of the form $\{\sigma_0 \subset \sigma_1 \subset \ldots \subset \sigma\}$ have the same local group equal to $P_{\sigma}$, where $P_{\sigma}$ is the local group of $G(\mQ)$ at $\sigma$. Thus blocks of $\mQ'$ are of the form $C_{\sigma}=\bigcup_{k}\{\sigma_0 \subset \sigma_1 \subset \ldots \subset \sigma_k \mid \sigma_k =\sigma \}$ and one can define a morphism  $\mQ \to \mR$ by $\sigma \mapsto  C_{\sigma}$. It is straightforward to check that it is an isomorphism and that so is the induced morphism $G(\mQ) \to G(\mR)$.
\end{proof}

\end{example}

\section{Reflection-like actions}\label{sec:reflikeactions}

In this section we introduce \emph{reflection-like actions}, which generalise the actions of reflection groups on Euclidean spaces. Our main application is the construction of new counterexamples to the strong form of Brown's conjecture regarding the equality between $\vcd G$ and $\ugd G$ (see \cite[ch.~2]{brownwall} or \cite[VIII.11]{brownco}):
\begin{conjBrown} Let $G$ be a virtually torsion-free group with $ \vcd G<\infty$. 

\begin{itemize}
\item[(i)] (weak form)  There is a contractible proper $G$-CW-complex of dimension $\vcd G$.
\item[(ii)] (strong form)  $\underline{\mathrm{gd}}G = \vcd G$.
\end{itemize}
\end{conjBrown}

Our counterexamples are similar to those of \cite{LePe}, where the desired group $G$ is a semidirect product of $W_L$ and $F$ where $W_L$ is a right-angled Coxeter group associated to a flag complex $L$ and $F$ is a finite group acting on $L$. However our  method of producing these counterexamples is different. In our case, we require the action of $F$ on $L$ to be reflection-like and rely on an application of Theorem \ref{thm:textrigidgen}.

To the best of our knowledge, the only known example of a reflection-like action that serves as a counterexample to the strong form of Brown's Conjecture is the action of $A_5$ on the $2$-skeleton of the Poincar{\'e} homology sphere (see, \cite[Example 1]{LePe}). In Example~\ref{ex:AS} we generalise this example. The reader may also look at the treatment of this example in \cite{PePry}, where the action is implicitly proven to be reflection-like.

\begin{definition}[Reflection-like action]\label{def:reflikeaction}
  Let $F$ be a group acting admissibly on a connected, flag simplicial complex $L$ of dimension $n \geqslant 1$, and let $Y\subseteq L$ be a strict fundamental domain for this action. We say that such an $F$-action is {\it reflection-like} if:

  \begin{enumerate}[label=(\roman*)]
    \item \label{enum:reflike1} The fundamental domain $Y$ is homeomorphic to the ball $B^n$,
    \item \label{enum:reflike2}  Every interior point of $Y$ has the same stabiliser, which we denote by $F_0$,
    \item \label{enum:reflike3} $F_0$ is a proper subgroup of the stabiliser of any point in $\partial B^n$.
  \end{enumerate} 
  Note that, in particular, part~\ref{enum:reflike3} implies that both the group $F$ and its action on $L$ are non-trivial.
\end{definition}

\begin{remark}\label{rem:easyreflectionlike}
  In the above definition, the assumptions on the action and on the complex $L$ are not very restrictive. Indeed, given an action of $F$ on a polyhedral complex $L$, by taking barycentric subdivision of $L$ one obtains an admissible action on a flag simplicial complex.
\end{remark}

\begin{lemma}\label{lem:reflikechains}
  Consider a reflection-like action of $F$ on $L$ with a strict fundamental domain $Y \subset L$. Let $\mathcal Q$ denote the poset of simplices of $Y$ ordered by the reverse inclusion and let $F(\mathcal Q)$ be the associated simple complex of groups (see Theorem~\ref{thm:scogisfund}). Then, the poset $\mathcal{Q}$ contains a block $C$ with local group $F_0$ such that:
  \begin{enumerate}
    \item $K_{C} = K \cong Y \cong B^{n}$,
    \item $K_{> C}\cong \partial(Y) \cong S^{n-1}$.
  \end{enumerate}
\end{lemma}

\begin{proof}
  The statement follows directly from the definition of a reflection-like action. Indeed, by Definition~\ref{def:reflikeaction}.\ref{enum:reflike2} the local group at any (open) simplex which does not lie on the boundary of $Y \cong B^{n}$, is necessarily equal to $F_0$. On the other hand, by Definition~\ref{def:reflikeaction}.\ref{enum:reflike3} the local group at any simplex on the boundary strictly contains $F_0$.
\end{proof}

\begin{definition}\label{def:coxsemidirect}
  Let $F$ be a finite group with a reflection-like action on a connected, compact, $n$-dimensional, flag simplicial complex $L$ with a strict fundamental domain $Y\subseteq L$. 
  Let $W_L$ be the right-angled Coxeter group associated to $L$. Then the $F$-action of $L$ induces an $F$-action on $W_L$. Let $G=W_L\rtimes F$ be the associated semi-direct product.
\end{definition}

In what follows, unless stated otherwise, let $F$, $L$, $Y$, and $G$ be as in Definition~\ref{def:coxsemidirect}.

\begin{proposition}\label{prop:coxeterisreflike}
  The group $G$ acts on Davis complex $\Sigma_{W_L}$ with strict fundamental domain and this action  is proper and reflection-like.
\end{proposition}

\begin{proof}
 	The group $G$ acts properly on Davis complex $\Sigma_{W_L}$ with a strict fundamental domain \cite[Lemma~3.5]{LePe}. One easily verifies that the fundamental domain is equal to $C(Y')$, the cone over the barycentric subdivision of $Y$. Since $Y \cong B^n$, we get that $C(Y') \cong B^{n+1}$ and thus part~\ref{enum:reflike1} of Definition~\ref{def:reflikeaction} is satisfied. For parts~\ref{enum:reflike2} and \ref{enum:reflike3} we need to identify $G$-stabilisers of the points in $C(Y')$. Recall that \[ C(Y')=\bigslant{Y' \times [0,1]}{(x,1) \sim (x',1)} \] and let $[x,t]$ denote the equivalence class of a point $(x,t) \in Y' \times [0,1]$. 
	
	\begin{enumerate}

    \item For the points in the interior of $C(Y')$, i.e.,\ points [x,t] where $x \in \mathrm{int}(Y')$ and $t \in (0,1)$ we have $\mathrm{Stab}_G[x,t]=F_0$ (where $F_0$ is the stabiliser of points in $\mathrm{int}(Y)$ with respect to the $F$-action on $L$). This establishes part~\ref{enum:reflike2} of Definition~\ref{def:reflikeaction}.

    \item We have three types of points on the boundary of $C(Y')$:

    \begin{enumerate}

      \item For the points $[x,0] $ where $x \in Y'$ the stabiliser $\mathrm{Stab}_G[x,t]$ is the Cartesian product of at least one generator of $W_L$ and the stabiliser of $x \in Y$ with respect to the $F$-action on $L$.

      \item For the points $(x,t)$ where $x \in \partial(Y')$ and $t \in (0,1)$, the stabiliser $\mathrm{Stab}_G[x,t]$ is equal to the stabiliser $x \in \partial(Y)$ with respect to the $F$-action on $L$.
      
      \item The stabiliser of the point $[x,1]$  is equal to the entire $F$. 

    \end{enumerate}

	Note that in each of the above cases, the stabiliser of $[x,t]$ strictly contains $F_0$. In case a) this follows from the fact that there is at least one generator of $W_L$ in the stabiliser, and in cases b) and c) this follows from the definition of a reflection-like action.
  Thus part~\ref{enum:reflike3} of Definition~\ref{def:reflikeaction} is satisfied, and therefore the $G$-action on $\Sigma_{W_L}$ is reflection-like.\qedhere
  
  \end{enumerate}

\end{proof}

\begin{lemma}\label{lem:reflikecoxeterdim}
 	Let $G(\mathcal Q)$ be a simple complex of groups associated to the $G$-action on $\Sigma_{W_L}$. Then $G$ is isomorphic to the fundamental group of $G(\mathcal Q)$ and  $$\dim D(B, G(\mathcal Q))=\dim D(K, G(\mathcal{Q}))=\ugd G=\ucd G=n+1.$$
\end{lemma}

\begin{proof} 
 	Since $\Sigma_{W_L}$ is simply connected, by Theorem~\ref{thm:scogisfund} we conclude that $G$ is isomorphic to the fundamental group of $G(\mathcal Q)$. The $G$-action on $\Sigma_{W_L}$ is proper and cocompact, and since $\Sigma_{W_L}$ is $\mathrm{CAT}(0)$, it follows that $\Sigma_{W_L}$ is a cocompact $G$-CW-model for $\underline{E}G$. Note that $G(\mathcal Q)$ is rigid, since all of its local groups are finite.  
  
 	Thus the assumptions of Theorem~\ref{thm:textrigidgen} are satisfied and we can use it to compute Bredon dimension of $G$. First note that since $\mathrm{dim} (\Sigma_{W_L})=n+1$, we get that $\underline{\mathrm{cd}}G \leqslant n+1$. Thus it suffices to show that  $\underline{\mathrm{cd}}G \geqslant n+1$. By Proposition~\ref{prop:coxeterisreflike} the $G$-action on $\Sigma_{W_L}$ is reflection-like and thus by Lemma~\ref{lem:reflikechains} the poset $Q$ contains a block $C$ such that:
  	\begin{enumerate}
    \item $K_C \cong C(Y') \cong B^{n+1}$,
    \item $K_{> C}\cong \partial(C(Y')) \cong S^n$.
  	\end{enumerate}
  	Since $H^{n+1}(B^{n+1}, S^n) \cong \mathbb{Z} \neq 0$, by Theorem~\ref{thm:textrigidgen}  we have that $\underline{\mathrm{cd}}G \geqslant n+1$. 
\end{proof}

\begin{lemma}\label{lem:reflikevirtualdim} 
  	If $H^n(L)=0$ then $\vcd G\leqslant n$.
\end{lemma}

\begin{proof}
	Since $G$ is a finite extension of $W_L$, we have that $\vcd G=\vcd W_L$. To prove that $\vcd W_L \leqslant n$, by \cite[Theorem~2]{Drani} it suffices to show that $H^n(Lk(\sigma, L))=0$ for every simplex $\sigma$ of $L$. For any non-empty simplex $\sigma$, the link $Lk(\sigma, L)$ is at most $(n-1)$-dimensional, and thus $H^n(Lk(\sigma, L))=0$. If $\sigma$ is empty, we have $Lk(\sigma, L)\cong L$ and by the assumption we have $H^n(L)=0$.
\end{proof}

The following theorem can be used to construct new cocompact counterexamples to the strong form of Brown's Conjecture.

\begin{theorem}\label{thm:reflikemain}
	Let $F$ be a finite group admitting a reflection-like action on a compact, connected, flag simplicial complex $L$  of dimension $n \geqslant 1$. Let $W_L$ be the right-angled Coxeter group associated to $L$ and $G=W_L\rtimes F$ be the associated semi-direct product. Suppose that $H^n(L)=0$. Then  
  	\begin{equation*}
    	\vcd G\leqslant n \quad \text{and} \quad \ucd G=n+1.
  	\end{equation*}
\end{theorem}

\begin{proof}
	The statement follows immediately from combining Lemmas~\ref{lem:reflikecoxeterdim} and~\ref{lem:reflikevirtualdim}.
\end{proof}

\subsection{Examples of reflection-like actions}

It remains to produce examples of groups satisfying the assumptions of Theorem~\ref{thm:reflikemain}. In every example discussed below, the underlying space admits an invariant polyhedral structure, which we will not specify (cf.\ Remark~\ref{rem:easyreflectionlike}).

We begin with the following two preparatory lemmas.

\begin{lemma}\label{lem:productreflike} 
	Suppose we have reflection-like actions of $F_1$ on an $m$-dimensional complex $L_1$ and of $F_2$ on an $n$-dimensional complex $L_2$. Then:

	\begin{enumerate}
    	\item The induced action of $F_1 \times F_2$ on $L_1 \times L_2$ is reflection-like. The fundamental domain is equal to the product of the respective fundamental domains and it is homeomorphic to $B^{m+n}$.

    	\item The induced action of $F_1 \times F_2$ on the join $L_1 \ast L_2$ is reflection-like. The fundamental domain is equal to the join of the respective fundamental domains and it is homeomorphic to $B^{m+n+1}$. 

	\end{enumerate}  
\end{lemma}

	The proof is straightforward and follows at once from the definition of a reflection-like action.

\begin{lemma}\label{lem:topdimcoh}
	Let $L_1$ be an $m$-dimensional finite complex and $L_2$ be an $n$-dimensional finite complex. Assume that either 
	\begin{enumerate} 
		\item \label{it:topcohvanishes1} $H^{m}(L_1)=0$, or
		\item \label{it:comprimetorfunctor} $H_m(L_1)=0$, $H_n(L_2)=0$, and $\mathrm{Tor}(H_{m-1}(L_1), H_{n-1}(L_2))=0$.
	\end{enumerate}
	Then $H^{m+n}(L_1 \times L_2)=0$ and $H^{m+n+1}(L_1 \ast L_2)=0$.
\end{lemma}

Note that the assumption $\mathrm{Tor}(H_{m-1}(L_1), H_{n-1}(L_2))=0$ is equivalent to torsion subgroups of $H_{m-1}(L_1)$ and $H_{n-1}(L_2)$ having coprime orders. 

\begin{proof}
	The claim follows easily from the K\"{u}nneth formula, the Universal Coefficients Theorem and the Mayer-Vietoris sequence for the join and the product.
\end{proof}

Note that Lemma~\ref{lem:productreflike} gives an easy way of producing new examples of reflection-like actions out of old ones, and Lemma~\ref{lem:topdimcoh} can be used to ensure that  top-dimensional cohomology of the product/join will vanish. In order to construct genuinely new examples with vanishing top-dimensional cohomology, we first construct examples that do have non-zero top-dimensional cohomology, and then combine them into products or joins and use Lemma~\ref{lem:topdimcoh} to ensure that the top-dimensional cohomology vanishes.

The summary of the constructed examples is presented in Table~\ref{table:1}.

\begin{table}[h!]

\centering 
\begin{tabular}{|c|c|c|c|c| } 
	\hline
	Example &$F$ & $L$ & $\mathrm{dim}(L)$ & $H^{\mathrm{dim}(L)}(L) = 0$? \\ \hline

 	\ref{ex:finitereflectiongroup}& $F\leqslant O(n+1)$, $F$ finite & $S^{n}$ & $n$ &  no \\ 

	\ref{ex:projectiveplane}& $(\mathbb{Z}/2)^{n} $& $\mathbb{R}P^{n}$ & $n$ & no \\ 

	\ref{ex:AS}& $PGL_2(q)$, $q=2^{a}$, $ a \geqslant 2$ & $L_q$ & $2$ & no, unless $q=4$ \\ 
	
	\ref{ex:dihedralmoore}& $D_k$ & $M_k$   & $2$ & no \\ 

	\ref{ex:joinASprojective}& $(\mathbb{Z}/2)^{n} \times PGL_2(q)$, $n$ even & $\mathbb{R}P^{n} \times L_q $ & $n+2$ & yes \\ 

	&  & $\mathbb{R}P^{n} \ast L_q $ & $n+3$ & yes \\ 

	\ref{ex:twocoprimemoorespaces}& $D_k \times D_l$, $k$ and $l$ coprime & $M_k \times M_l$   & $4$ & yes \\

	& & $M_k \ast M_l$   & $5$ & yes \\

	\ref{ex:projectiveplanemoorespace}& $(\mathbb{Z}/2)^{n} \times D_k$, $n$ even,  $k$ odd & $\mathbb{R}P^{n}\times M_k$   & $n+2$ & yes \\

	&  & $\mathbb{R}P^{n}\ast  M_k$   & $n+3$ & yes \\

	 \hline

\end{tabular}

\vspace{10pt}

\caption{Examples of reflection-like actions, together with an indication whether they satisfy the assumptions of Theorem~\ref{thm:reflikemain}.}
\label{table:1}
\end{table}

\begin{example}[Finite reflection group]\label{ex:finitereflectiongroup}
	Let $F \leqslant O(n)$ be a finite subgroup generated by  orthogonal reflections across hyperplanes in $\mathbb{R}^n$ (see \cite[Chapter~6]{Davbook}). Then the induced action of $F$ on the unit sphere $S^{n-1}\subset \mathbb{R}^n$ is reflection-like.
\end{example}

\begin{example}\label{ex:projectiveplane}
	Consider the action of $\mathbb{Z}/2$ on $\mathbb{R}$ given by $x \mapsto -x$ and consider the product action of $(\mathbb{Z}/2)^n$ on $\mathbb{R}^n$. Factoring out the action of the antipodal map $\iota \in (\mathbb{Z}/2)^n$, we obtain an action of  $(\mathbb{Z}/2)^n/\langle \iota \rangle \cong (\mathbb{Z}/2)^{n-1}$ on the real projective space $\mathbb{R}P^{n-1}$. One easily verifies that this action is reflection-like, with the quotient being an $(n-1)$-simplex.
\end{example}

The above example is a special case of the so-called \emph{small cover} of Davis and Januszkiewicz \cite{DJ}, which is an $n$-dimensional manifold together with a reflection-like action of $(\mathbb{Z}/2)^n$ whose quotient is isomorphic to an $n$-dimensional simple polytope.

\begin{example}[Aschbacher-Segev]\label{ex:AS}
  	We outline a construction of a reflection-like action of the group $F=PGL_2(q)$ for $q=2^a, a\geqslant 2$ on a flag $2$-complex $L =L_q$  in order to illustrate the underlying simple complex of finite groups $F(\mathcal{Q})$. For more details we refer to \cite[Section 9]{AS}. 

  	For the $1$-skeleton $L_q^{(1)}$ take the barycentric subdivision of the complete graph on the projective line of  $q+1$ points $v_1, \dots, v_{q+1}$ with the natural action of $F$. Fix a single conjugacy class $\mathcal C$ of cycles of order $q+1$ in $F$. Every cycle of order $q+1$ is conjugate to its inverse. Therefore, there are $q(q-1)/2$ pairs of $(q+1)$-cycles $(\sigma_i, \sigma^{-1}_i)$ in $\mathcal C$. Define $L_q$ by attaching that many $(q+1)$-gons using the cycles $\sigma_i$ to describe the attaching maps. Each $2$-cell becomes a cone on its subdivided $(q+1)$-gonal boundary where $\sigma_i$ acts by fixing the cone point. The $2$-simplices of $L_q$ are $q(q^2-1)$ right-angled triangles on which $F$ acts simply transitively. Each one is a strict fundamental domain. Let $Y$ be such a fundamental domain that contains a vertex $v_j$ whose stabiliser is the Borel subgroup $B$ of upper triangular matrices in $F$.

	Figure~\ref{fig:reflike} a) shows the fundamental domain $Y$ together with local groups at cells. Figure~\ref{fig:reflike} b) shows the fundamental domain $C(Y')$ for the associated action of $W_L \rtimes F$ on $\Sigma_{W_L}$ together with local groups at vertices. Local groups at cells are given by the respective intersections of local groups at vertices. The generators of $W_L$ corresponding to vertices of $Y$ are denoted by $s_1,s_2,$ and $s_3$.\smallskip

	$(\ast)$ For small values of $q$, the complex $L_q$ is known to be $\mathbb{Q}$-acyclic, with first homology either trivial or elementary abelian of order $r^{q-1}$, where $r$ is an odd prime. For $q=4$, the complex $L_q$ is homeomorphic to the Poincar{\'e} dodecahedron, and hence it is acyclic. 
\end{example}

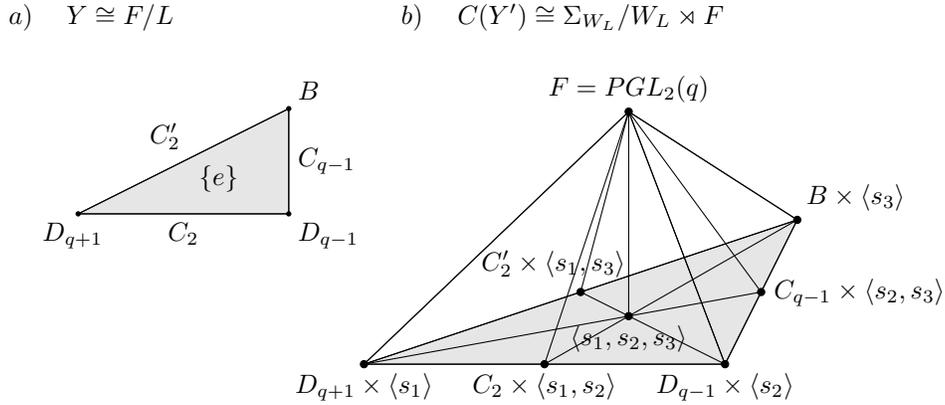
\begin{figure}[!h]
\centering
\begin{tikzpicture}[scale=0.4]

\definecolor{torange}{RGB}{250,125,000}
\definecolor{lorange}{RGB}{250,200,075}
\definecolor{vlgray}{RGB}{230,230,230}
\definecolor{tpurple}{RGB}{100,150,00}
\definecolor{lpurple}{RGB}{150,200,050}
\definecolor{blue}{RGB}{000,50,200}

\begin{scope}[shift={(-5.5,5)},scale=0.7]

\draw[fill=vlgray] (-5,0)--(5,0)--(5,5)--(-5,0);

\node   at (0.333*5,0.333*5)   {$\{e\}$};

\node [below]  at (-5.25,0)   {$D_{q+1}$};

\node [below ]  at (0,0)   {$C_2$};

\node [below right]  at (5,0)   {$D_{q-1}$};

\node [right]  at (5,2.5)   {$C_{q-1}$};

\node [above right]  at (5,5)   {$B$};

\node [above left]  at (0.3,2.6)   {$C_2'$};

\draw (-5,0)--(5,0)--(5,5)--(-5,0);

\draw[fill]  (-5,0)  circle [radius=0.1];
\draw[fill]  (5,0)  circle [radius=0.1];
\draw[fill]  (5,5)  circle [radius=0.1];

\end{scope}

\begin{scope}[shift={(6.5,0)},scale=1.2]

\draw[fill=vlgray] (-5,0)--(5,0)--(7,4)--(-5,0);


\node [above]  at (-14.5,9)   {$a)$};

\node [above right]  at (-13.5,9)   {$Y \cong F/L$};

\node [above]  at (0.333*7-6,9)   {$b)$};

\node [above right]  at (0.333*7-5,9)   {$C(Y') \cong \Sigma_{W_L}/ W_L \rtimes F$};

\node [above]  at (0.333*7,7)   {$F=PGL_2(q)$};

\node [below ]  at (-5,0)   {$D_{q+1} \times \langle s_1 \rangle$};

\node [below ]  at (0,0)   {$C_2 \times \langle s_1,s_2 \rangle$};

\node [below]  at (5,0)   {$D_{q-1} \times \langle s_2 \rangle$};

\node [right]  at (6.1,2)   {$C_{q-1} \times \langle s_2,s_3 \rangle$};

\node [above right]  at (7,4)   {$B \times \langle s_3 \rangle$};

\node [below]  at (0.333*7,0.333*4)   {$\langle s_1,s_2, s_3 \rangle$};

\draw (0.333*7, 7)--(0.333*7,0.333*4);

\draw (-5,0)--(5,0)--(7,4)--(-5,0);

\draw (0,0)--(7,4);
\draw (5,0)--(1,2);
\draw (-5,0)--(6,2);

\draw (0.333*7, 7)--(0,0);
\draw (0.333*7, 7)--(7,4);
\draw (0.333*7, 7)--(5,0);
\draw (0.333*7, 7)--(1,2);
\draw (0.333*7, 7)--(-5,0);
\draw (0.333*7, 7)--(6,2);

\draw (-5,0)--(0.333*7,7);
\draw (5,0)--(0.333*7,7);
\draw (7,4)--(0.333*7,7);

\draw[fill]  (0,0)  circle [radius=0.1];

\draw[fill]  (1,2)  circle [radius=0.1];
\draw[fill]  (6,2)  circle [radius=0.1];
\draw[fill]  (0.333*7, 0.333*4)  circle [radius=0.1];

\draw[fill]  (-5,0)  circle [radius=0.1];
\draw[fill]  (5,0)  circle [radius=0.1];
\draw[fill]  (7,4)  circle [radius=0.1];

\draw[fill]  (0.333*7, 7)  circle [radius=0.1];

\node [above left]  at (2.5,2.15)   {$C_2' \times \langle s_1,s_3 \rangle$};

\end{scope}

\end{tikzpicture}
\caption{Fundamental domains $Y$ (a) and $C(Y')$ (b) together with stabilisers of cells and vertices respectively.}
\label{fig:reflike}
\end{figure} 

\begin{example}[Dihedral group acting on a Moore space]\label{ex:dihedralmoore}
	For a natural number $k\geqslant 2$, let $M_k$ denote the Moore space $M(\mathbb{Z}/k, 1)$, i.e.,\ a space obtained by attaching a disk to a circle along the map of degree $k$. Thus we have $\widetilde{H}_1(M_k) \cong \mathbb{Z}/k$ and $\widetilde{H}_i(M_k)=0$ for all $i \neq 1$. We will describe a reflection-like action of the dihedral group $D_k$ on $M_k$. Recall that $D_k$ is generated by two reflections $s$ and  $t$ and their product $st$ is a rotation of order $k$.

	Consider the standard action of $D_k$ on a $k$-gon and the reflection action of $D_k/ \langle st\rangle \cong \mathbb{Z}/2$ on a circle, both shown in Figure~\ref{fig:moore} a) (note that both actions reverse the orientation of the edges). The attaching map of the boundary of the $k$-gon is equivariant with respect to the homomorphism $D_k \to D_k/ \langle st\rangle \cong \mathbb{Z}/2$, and thus we get a well-defined action of $D_k$ on $M_k$. One easily checks that this action has a strict fundamental domain, which is a triangle. The fundamental domain together with its cell stabilisers is shown in Figure~\ref{fig:moore} b). By analysing the stabilisers, we conclude that the action of $D_k$ on $M_k$ is reflection-like.

	We remark that in this setting $M_k$ is homeomorphic to the Basic Construction $D(\abs{\mQ}, G(\mQ), \psi)$, where $G(\mQ)$ is a simple complex of groups associated to the $D_k$-action on $M_k$, and $\psi \colon G(\mQ) \to D_k$ is a simple morphism induced by sending all three vertex groups $D_k$ into $D_k$ via the identity map.

	Finally, observe that for $k=2$ in the above construction, $D_2 \cong \mathbb{Z}/2\times \mathbb{Z}/2$ is an isometry group of a $2$-gon and $M_k$ is equivariantly homeomorphic to the real projective plane $\mathbb{R}P^{2}$ appearing in Example~\ref{ex:projectiveplane}.
\end{example}

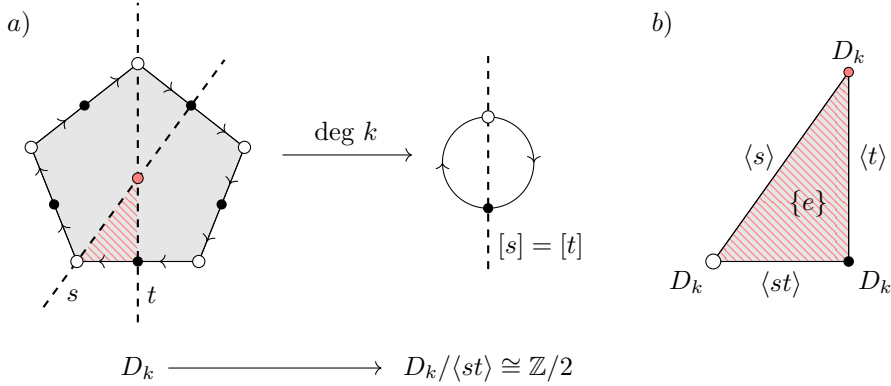
\begin{figure}[!h]
\centering
\begin{tikzpicture}[scale=0.45]

\definecolor{torange}{RGB}{250,125,000}
\definecolor{lorange}{RGB}{250,200,075}

\definecolor{vlgray}{RGB}{230,230,230}
\definecolor{tpurple}{RGB}{100,150,00}
\definecolor{lpurple}{RGB}{150,200,050}

\definecolor{blue}{RGB}{000,50,200}

\node   at (-24.5,7)   {$a)$};

\node   at (-5.5,7)   {$b)$};


\begin{scope}[scale=0.8]

\draw[fill=vlgray] (-5,0)--(0,0)--(0,7)--(-5,0);

\draw[ pattern= custom north west lines,hatchcolor=red!50] (-5,0)--(0,0)--(0,7)--(-5,0);

\node   at (-1.5,2.15)   {$\{e\}$};

\node [below left]  at (-5,0)   {$D_{k}$};

\node [below ]  at (-2.5,0)   {$\langle st \rangle$};

\node [below right]  at (0,0)   {$D_{k}$};

\node [above right]  at (0,3)   {$\langle t \rangle$};

\node [above]  at (0,7)   {$D_k$};

\node [above left]  at (-2.35,3)   {$\langle  s \rangle$};

\draw[fill=white]  (-5,0)  circle [radius=0.275];

\draw[fill]  (0,0)  circle [radius=0.175];
\draw[fill=red!50]  (0,7)  circle [radius=0.175];

\end{scope}


\begin{scope}[shift={(-21,0)},scale=0.9]

  \draw[fill=vlgray] (-2,0)--(-0,0)--(2,0)--(2.75,1.875)--(3.5,3.75)--(1.75,5.125)--(0,6.5)-- (-1.75,5.125)--(-3.5,3.75)--(-2.75,1.875)--(-2,0);
\draw[vlgray, pattern= custom north west lines,hatchcolor=red!50] (-2,0)--(0,0)--(0,2.735)--(-2,0);

\begin{scope}[yscale=1,xscale=-1]
\draw[->-] (-2,0)--(-0,0);

\draw[->-] (-0,0)--(2,0);
\draw[->-] (2,0)-- (2.75,1.875);
\draw[->-] (2.75,1.875)--(3.5,3.75);
\draw[->-] (3.5,3.75)--(1.75,5.125);
\draw[->-] (1.75,5.125)--(0,6.5);
\draw[->-] (0,6.5)-- (-1.75,5.125);
\draw[->-] (-1.75,5.125)--(-3.5,3.75);
\draw[->-] (-3.5,3.75)--(-2.75,1.875);
\draw[->-] (-2.75,1.875)--(-2,0);

\end{scope}

\draw[thick, dashed] (0,-2)--(0,8.5);

\node [right]  at (0,-1.1)   {$t$};

\draw[thick, dashed]  (-2-0.732*1.5,0-1.5)--(-2,0)--(1.75,5.125)--(1.75+0.732*1.5,5.125+1.5);

\draw[fill=red!50]  (0,2.735)  circle [radius=0.175];

\node [right]  at (-2.65,-1.1)   {$s$};

\node   at (0,-3.5)   {$D_k$}; 
\draw[->] (1,-3.5)--(8,-3.5);

\node   at (11.5,-3.5)   {$D_k/\langle st \rangle \cong \mathbb{Z}/2$}; 

\draw[fill=white]  (-2,0)  circle [radius=0.2];
\draw[fill=white]  (2,0)  circle [radius=0.2];
\draw[fill=white]  (3.5,3.75) circle [radius=0.2];
\draw[fill=white]  (0,6.5) circle [radius=0.2];
\draw[fill=white]  (-3.5,3.75) circle [radius=0.2];

\draw[fill]  (-0,0)  circle [radius=0.15];
\draw[fill]  (2.75,1.875)  circle [radius=0.15];
\draw[fill]  (1.75,5.125) circle [radius=0.15];
\draw[fill]  (-2.75,1.875)  circle [radius=0.15];
\draw[fill]  (-1.75,5.125) circle [radius=0.15];


\begin{scope}[shift={(11.5,0)}]

\draw[->] (-6.75,3.5)--(-2.5,3.5);
\node [above]  at (-4.8,3.5)   {$\text{ deg } k$};

 \draw[thick, dashed] (0,-0.25)--(0,6.75);

 \node[right]   at (0,0.5)   {$[s]=[t]$};

\draw  (0,3.25)  circle [radius=1.5];

\draw[fill=white]  (0,4.75) circle [radius=0.2];

\draw[fill]  (0,1.75)  circle [radius=0.15];

\draw[->-] (-1.5,3.25)--(-1.5,3.2501);
\draw[->-] (1.5,3.25)--(1.5,3.2499);

\end{scope}

\end{scope}

\end{tikzpicture}
\caption{a) Reflection like action of $D_k$ on a Moore space $M_k$. b) The fundamental domain together with local groups at cells.}
\label{fig:moore}

\end{figure} 

We are ready now to construct new counterexamples to the strong form of Brown's Conjecture. 

\begin{example}\label{ex:joinASprojective} 
	Let $L_q$ be a complex as in Example~\ref{ex:AS} satisfying $(\ast)$. For an even integer $n$, consider the induced reflection-like actions of the product $(\mathbb{Z}/2)^{n} \times PGL_2(q)$ on the product $\mathbb{R}P^{n} \times L_q$  and on the join $\mathbb{R}P^{n} \ast L_q$. 

 	Since $H_n(\mathbb{R}P^{n})=0$, $H_{n-1}(\mathbb{R}P^{n})=\mathbb{Z}/2$, $H_2(L_q)=0$ and $H_1(L_q)$ is either trivial or elementary abelian of order being a power of an odd prime, by Lemma~\ref{lem:topdimcoh} we conclude that $H^{n+2}(\mathbb{R}P^{n} \times L_q)=0$ and $H^{n+3}(\mathbb{R}P^{n} \ast L_q)=0$. 
\end{example}

\begin{example}\label{ex:twocoprimemoorespaces}
	Consider $M_k$ and $M_l$ such that $k$ and $l$ are coprime. By Lemma~\ref{lem:topdimcoh} we get that $H^{4}(M_k \times M_l)=0$ and $H^{5}(M_k \ast M_l)=0$ (in fact $M_k \ast M_l$ is contractible).
\end{example}

\begin{example}\label{ex:projectiveplanemoorespace}
	For an even integer $n$ and an odd integer $k$ consider the action of $(\mathbb{Z}/2)^{n}$ on the real projective space $\mathbb{R}P^{n}$, and the action of $D_k$ on the Moore space $M_k$. By Lemma~\ref{lem:topdimcoh} we have $H^{n+2}(\mathbb{R}P^{n} \times M_k)=0$ and $H^{n+3}(\mathbb{R}P^{n} \ast M_k)=0$.
\end{example}

\begin{remark} In contrast to Example~\ref{ex:AS} (and \ref{ex:joinASprojective}), Examples~\ref{ex:twocoprimemoorespaces} and \ref{ex:projectiveplanemoorespace} are particularly simple in terms of algebraic structure of groups and cellular structure of complexes. The smallest group appearing in these examples is the product $D_2 \times D_3 \cong (\mathbb{Z}/2)^{2} \times  S_3$.
\end{remark}

\section{Final remarks and open questions}\label{sec:questions}

	Let $X$ be a $G$-CW-complex. We say that a $G$-CW-subcomplex $Y$ is a {\it spine} of $X$ if it is an equivariant deformation retract of $X$. When $X$ is a model for $\EFG$,  then so is $Y$ and $\dim (Y)\geqslant \fgd G$.   Spines of minimal dimension (so equal to $\fgd G$) have been constructed, for example, for the actions of certain arithmetic groups such as $\mbox{SL}(n, \Z)$ on the symmetric space \cite{Ash}, the actions of the outer automorphism groups of free groups on the Outer space \cite{Vog}, and the actions of the mapping class groups of punctured  surfaces on the Teichm\"{u}ller space \cite{Har}.

	\begin{ques}\label{ques:spine}
	  Suppose a group $G$ acts on a $\mathrm{CAT}(0)$ polyhedral complex $X$ with a strict fundamental domain. Denote by $\mF$  the family generated by the stabilisers. Suppose the associated complex of groups $G(\mathcal Q)$ is thin.  Can  $D(B, G(\mathcal{Q}))$ be constructed as a spine of $X$ of the lowest possible dimension equal to $\fgd G$?
	\end{ques}

  Theorem \ref{thm:deformationretraction} tells us  that the answer is yes if  $ \fcd G\leqslant 1$. Also by Theorem \ref{thm:intromain}, we know that $\dim D(B, G(\mathcal{Q}))=\fgd G$ and $D(B, G(\mathcal{Q}))$ is $G$-homotopy equivalent to $X$.  The question whether $D(B, G(\mathcal{Q}))$ can be constructed as an equivariant deformation retract of $X$ is open in general. In \cite{PePry}, we isolate a condition on a finite polyhedra which we call {\it subconical}. It is open whether every finite polyhedron is subconical. If this is the case, then a generalisation of  Proposition 4.7 of \cite{PePry} to thin simple complexes of groups gives an affirmative answer to this question.

  \begin{ques}\label{ques:cat(0)metric} 
    Does $D(B, G(\mathcal{Q}))$ attain the CAT(0) dimension of the group $G$?
  \end{ques}

  In many cases of interest, such as Coxeter groups or groups acting on buildings, the associated standard development $D(K, G(\mathcal{Q}) )$ supports a $G$-invariant $\mathrm{CAT}(0)$ metric. Therefore it is natural to ask whether Bestvina complex supports such a metric as well, or whether the dimension of Bestvina complex is equal to the \emph{$\mathrm{CAT}(0)$ dimension of the group for the family $\mathcal F$}. The latter is defined as the minimal dimension of a  model for $\EFG$ that supports a $G$-invariant $\mathrm{CAT}(0)$ metric.

  There are simple complexes of groups where the corresponding Bestvina complex does not admit any $G$-invariant piecewise linear $\mathrm{CAT}(0)$ metric \cite{Pry}. Moreover, we suspect that these examples also have $\mathrm{CAT}(0)$ dimension strictly larger than the Bredon cohomological dimension. The above examples are the right-angled Coxeter groups (or graph products) associated to certain $2$-dimensional contractible but non-collapsible complexes. Consequently, the methods used to show the lack of $\mathrm{CAT}(0)$ metric do not carry through to higher dimensions, and to the best of our knowledge the question is open in all dimensions greater than $2$.  

  The question is especially interesting when $\mF$ is the family of all finite subgroups. In this case, the metric structure of $\underline{E}G$ can be used to study numerous features of $G$ e.g., by considering the visual boundary of $\underline{E}G$. Note that the positive answer to that question, combined with Example \ref{ex:twocoprimemoorespaces} (or \ref{ex:projectiveplanemoorespace}), would result in a group of $\mathrm{CAT}(0)$ dimension four, whose finite-index overgroup has $\mathrm{CAT}(0)$ dimension equal to five.

\begin{ques}\label{ques:weakbrown}
  Are the groups $G$ constructed in Example \ref{ex:twocoprimemoorespaces} or \ref{ex:projectiveplanemoorespace} also counterexamples to the weak form of Brown's conjecture?
\end{ques}

The weak form of Brown's conjecture is open in all dimensions except when $\vcd G = 2$ \cite{LePe}.  A natural place to look  for counterexamples are the groups that disprove the strong form of Brown's conjecture. Yet, most such groups $G$ are known to act properly on a contractible complex of dimension $\vcd G$. Take for example $G=W_L\rtimes F$. If $L$ is contractible (see \cite[\S 5]{LePe} for examples), then there is a contractible subcomplex $Y$ of $\Sigma_{W_L}$ of dimension $\vcd G$ on which $G$ acts properly. The subcomplex $Y$ can be obtained by applying the Basic Construction to $L'$ instead of $CL'$. Similarly, the finite extensions of Bestvina-Brady groups  constructed in \cite{LeNu} or \cite[3.6]{marper}, cannot be counterexamples to the weak form of the conjecture, because they act properly on the level sets of the Morse function which in these examples are contractible.

\end{document}